\documentclass[12pt]{amsart}

\usepackage{stackrel}
\usepackage{amssymb}
\usepackage{amsmath}
\usepackage{amsthm}
\usepackage[pdftex]{graphicx}
\usepackage{mathtools}
\usepackage{pdfsync}
\usepackage{graphicx}
\usepackage{color}
\usepackage{tikz-cd}
\usepackage{tikz}
\usetikzlibrary{positioning}
\usetikzlibrary{decorations.text}
\usetikzlibrary{decorations.pathmorphing}
\usetikzlibrary{matrix}

\usepackage[all]{xy}

\newtheorem{theorem}{Theorem}[section]
\newtheorem{lemma}[theorem]{Lemma}
\newtheorem{proposition}[theorem]{Proposition}
\newtheorem{corollary}[theorem]{Corollary}

\theoremstyle{definition}

\newtheorem{definition}[theorem]{Definition}
\newtheorem{example}[theorem]{Example}

\theoremstyle{remark}

\newtheorem*{remark}{Remark}
\newtheorem*{Question}{Question}
\numberwithin{equation}{section}

\DeclarePairedDelimiterX\Set[2]{\lbrace}{\rbrace}
 { #1 \,\delimsize|\,\mathopen{} #2 }

\def\({\left(}
\def\){\right)}

 \newcommand{\Z}{\mathbb{Z}}
  \newcommand{\N}{\mathbb{N}}

\title[Top.  Fact.  of Zero Dim. Dyn. Sys.]{Topological Factoring of Zero Dimensional  Dynamical Systems}

\author{N. Golestani, M. Hosseini, H. Yahya Oghli}

\address{Department of Pure Mathematics, Faculty of Mathematical Sciences, 
  Tarbiat Modares University,
 Tehran\\ Iran}
 \email{n.golestani@modares.ac.ir}
\address{School of Mathematical Sciences, Queen Mary University of London, Mile End Road
London E1 4NS, UK }
\email{m.hosseini@qmul.ac.uk}
\address{Department of Pure Mathematics, Faculty of Mathematical Sciences, 
  Tarbiat Modares University,
 Tehran\\ Iran}
 \email{hamed.yahyaoghli@modares.ac.ir}

\begin{document}
\maketitle
\begin{abstract}
 We show that every topological factoring between two zero dimensional dynamical systems  can be represented by a sequence  of  morphisms  between the levels of the associated ordered Bratteli diagrams.  Conversely, we will prove that given an ordered Bratteli diagram $B$ with a continuous Vershik map on it,  every sequence of morphisms between levels of $B$ and $C$, where  $C$ is another ordered Bratteli diagram with continuous Vershik map,  induces a topological factoring if and only if  $B$ has a unique infinite min path. We present a method to construct various examples of ordered premorphisms between two decisive Bratteli diagrams such that the induced maps between the two Vershik systems are not topological factorings. We provide sufficient conditions for the existence of a topological factoring from an ordered premorphism.  Expanding on the modelling of factoring, we generalize the Curtis-Hedlund-Lyndon theorem to represent factor maps between two zero dimensional dynamical systems through sequences of sliding block codes. 
 \end{abstract}
\section{Introduction}
Topological dynamical systems on zero dimensional spaces, specifically Cantor sets, have been extensively researched over the past decades. Pioneering theorems, such as Jewett-Krieger's, support these studies, stating that every ergodic system on a probability Lebesgue space is isomorphic to a uniquely ergodic minimal system on a Cantor set \cite{J, K}.

Some of the mostly used ``models" in studying   zero dimensional dynamical system  are transformations acting on the shift spaces or  Bratteli-Vershik systems on ordered Bratteli diagrams (see Subsection~\ref{versh}). The notion of Bratteli diagram was  introduced in operator algebras and then by the celebrated work of R. Herman, I. Putnam and C. Skau  \cite{hps92}, became a tool for studying zero dimensional systems and absorbed a lot of investigations by people in symbolic dynamics and operator algebras  \cite{B, don, dm, gps95, gw95, ormes, Tak1, Tak2}. Creating a Bratteli diagram for a zero dimensional system is based on the existence of  {\it Kakutani-Rokhlin (K-R) partitions}  for them.  That is a  certain  union of disjoint clopen sets that covers the space $X$. Existence of  a K-R partition   is equivalent to the existence of  a {\it complete $T$-section}, i.e.  a clopen set $U$ that  hits every orbit:  $$\bigcup_{n\in\mathbb Z}T^n(U)=X.$$
This equality together with  the compactness of the space $X$  leads to  a K-R partition $\mathcal P$. Having a sequence of K-R partitions ({\it K-R system}), say $\{\mathcal P_n\}_{n\geq 0}$, in such a way that for every $n>0$ the union of the top levels of $\mathcal P_n$ is contained in the union of  the top levels of $\mathcal P_{n-1}$, generates a closed set $W$ that all its clopen neighbourhoods are complete $T$-sections \cite{Poon}. In the first version of this paper (posted on arXiv) we called such a  closed set, {\it a weak basic set}. However, soon after  we learned that the same concept has been considered in a recent paper of T. Shimomura   in \cite{Tak3} as  {\it a quasi-section}. Not generating different terminologies for the same concept, in the sequel we use the  phrase quasi-section instead of weak basic set.  

The existence of  quasi-sections  for a zero dimensional dynamical system $(X,T)$  was firstly established for minimal systems on Cantor sets by I. Putnam in \cite{putnam89} and then for Cantor essentially minimal systems  in \cite{hps92}, where it was shown that every singleton $\{x\}$ that  $x$ belongs to the unique minimal subsystem, is a quasi-section set. For Cantor aperiodic systems, K. Medynets in \cite{med} defined the concept of a "basic set" (a quasi-section that intersects every orbit at most once) and proved that every Cantor aperiodic system possesses a basic set.  In the more general scenario of zero-dimensional dynamical systems, when the set of aperiodic points is dense, the system's realization through a sequence of K-R partitions can be inferred from the findings of T. Downarowicz and O. Karpel in \cite{dk16, dk17}, where they characterized such systems as "array systems". The endeavour to represent zero-dimensional dynamical systems as K-R systems, with the convergence of top levels towards a quasi-section, was ultimately accomplished by T. Shimomura in \cite{Tak2}, where the notion of "graph-covering models" was introduced and utilized to establish that every zero-dimensional dynamical system can be modelled as a Bratteli-Vershik system.

 It turns out that  if $(X,T)$ is a zero dimensional dynamical system, then it  can be modelled by a  {\it Vershik homeomorphism}  $T_B$  on  the compact space of all (partially ordered) infinite paths of an ordered Bratteli diagram $X_B$ that the set of its infinite maximal paths is associated  to a  quasi-section  \cite{Tak3} (see also Proposition~\ref{perfect}).  The Vershik map is initially defined as $T_B:X_B\setminus X_B^{\rm max}\rightarrow X_B\setminus X_B^{\rm min}$ (see Definition~\ref{order}) and then 
the domain of $T_B$ may be extended to the whole space to have $\bar{T}_B: X_B\rightarrow X_B$  that $\bar{T}_B$ maps $X_B^{\rm max}$  into $X_B^{\rm min}$ (that we call it {\it natural extension}). The ordering of such a diagram  in which the Vershik map can be extended  to the whole space is called {\it perfect} as was defined by   S. Bezuglyi and R. Yassawi in \cite{BY, BY1}. 
The extension is unique if and only if $(X,T)$ has a dense set of aperiodic points  \cite[Theorem 3.1]{dk17}. Such ordered Bratteli diagrams are called {\it decisive} and  
examples of  perfect orderings on Bratteli diagrams that are not decisive were provided in  \cite{dk16, dk17}. 
 
In this paper,  we study {\it topological factoring}  between two zero dimensional dynamical systems. We generalize the  ``modelling" introduced  by M. Amini, G. Elliott and N. Golestani in \cite{aeg21},  of  topological factoring between two Cantor minimal systems, to the topological factoring of (any) two zero dimensional systems  with   respect to their K-R system realizations.
In \cite{aeg21} the authors proved that having topological factoring from a minimal Cantor system $(X,T)$ onto another  minimal system on  a Cantor set, $(Y,S)$, creates a sequence of (local) morphisms between the sets of vertices of the levels of Bratteli diagrams $C=(W, E', \leq')$ and $B=(V,E,\leq)$  associated to  the two systems respectively.  That is $f:B\rightarrow C$ which is a sequence $f=\{F_i\}_{i\geq 0}$, $F_i:V_i\rightarrow W_{f_i}$, for an  increasing  sequence $\{f_i\}_{i\geq 0}$ of natural numbers together with a partial  ordering on each $F_i$ (see Definition \ref{def61} for details).  The sequence $\{F_i\}_{i\geq 0}$ makes edges between levels of the two diagrams with some orderings on them. This sequence  is called an {\it ordered premorphism} and in \cite{aeg21} it is proved that the existence of such a sequence between two (essentially simple) Bratteli diagrams  is equivalent to the existence of a topological factoring between the two Bratteli-Vershik systems:
 $$\alpha=\mathcal V(f):(X_C,T_C)\rightarrow (Y_B, S_B)$$
 induced by an ordered premorphism $f:B\rightarrow C$. 
 In fact, the ordered premorphism determines exactly where   in $(Y_B, S_B)$ an infinite path from $(X_C, T_C)$ will be mapped  under $\alpha$.  
 As an application of this interpretation of topological factoring between two Bratteli-Vershik systems, in \cite{gh18} it was proved that every topological factor of a {\it finite topological rank} essentially minimal dynamical  systems on a Cantor set,  is of finite topological rank. Using symbolic interpretation of ordered premorphisms, in \cite{bastian} the author proved that if $(X_C, T_C)$ is minimal and $S_B$ is a shift map,  then the topological rank of $(Y_B, S_B)$ is at most equal to the topological rank of $(X_C, T_C)$.

Generalization of  this theory of the one-to-one correspondence between ordered premorphisms and topological factorings is investigated here. In other words, as every zero dimensional system has a non-trivial Bratteli-Vershik representation \cite{Tak2}, we obtain a Bratteli-Vershik representation (called ordered premorphism) for every factor map between two zero dimensional systems.

\begin{theorem}\label{fact2}
Let $(X,T)$ and $(Y,S)$ be two zero dimensional dynamical systems and $X_0$ and $Y_0$ be quasi-sections for $X$ and $Y$, respectively. If $\alpha: (X,T)\rightarrow (Y,S)$ is a topological factoring  such that $\alpha(X_0)\subseteq Y_0$ then for any perfect diagrams $C$ and $B$ that are B-V realizations of  $(X,T,X_0)$ and $(Y,S,Y_0)$ respectively, there exists 
a unique (up to equivalence) 
ordered premorphism $f:B\rightarrow C$ such that $\mathcal V(f)=\alpha$.
\end{theorem}
 \medskip

To prove the existence part of the above  theorem, we   present the theorem of Shimomura  \cite[Theorem~1.1]{Tak3} in the frame of Lemma~\ref{finer} and Proposition ~\ref{perfect} to be able to  construct the ordered premorphism $f$ inductively for all levels of $B$ such that $\mathcal V(f)=\alpha$ as in  the  method for the proof of \cite[Proposition 4.6]{aeg21}. The proof of the uniqueness, is done in a more general setting in Proposition \ref{unique1}.
\medskip

The converse of Theorem~\ref{fact2} is not  true. Having an ordered Bratteli diagram $B$ with at least two infinite min paths, one can construct an ordered Bratteli diagram $C$ and an ordered premorphism $f:B\rightarrow C$ such that $\mathcal V(f)$ is not  topological factoring. We will prove in Section 2, that $\mathcal V(f)$ is onto and continuous  but in Section 3,  we will see that the equivariant equation $\mathcal V(f)\circ T_C(x)=T_B\circ \mathcal V(f)(x)$ may fail for $x\in X_C^{\rm max}$.
However, the converse of Theorem~\ref{fact2} holds to some extent as
we prove in the following theorem. To have the results in  full generality, we say that an ordered Bratteli diagram is {\it semi-decisive} if it admits a continuous surjective extension $\bar{T}_B$ of the Vershik map $T_B$ to  $X_B$.  

\begin{theorem}\label{main}
Let $B$ be a semi-decisive ordered Bratteli diagram  such that its Vershik map  has a natural extension $\bar{T}_B$  to $X_B$. The following statements are equivalent:
\begin{enumerate}
\item \label{one} for every semi-decisive ordered Bratteli diagram $C$  with its natural extension $\bar{T}_C$ and every ordered premorphism $f:B\rightarrow C$, the induced map $\mathcal V(f): X_C\rightarrow X_B$ is a topological factoring.
\item \label{two} $B$ has a  unique infinite min path.
\end{enumerate}

\end{theorem}

As a corollary, if $B$ is simple and its ordering is perfect, then for every  diagram $C$ with perfect ordering and every ordered premorphism $f:B\rightarrow C$, the induced map $\mathcal V(f)$ makes a topological factoring if and only if  $B$ is proper (Corollary~\ref{proper}). 

To prove Theorem \ref{main}, we introduce a method for constructing a (semi-)decisive Bratteli diagram $C$ associated with a given non-proper (semi-)decisive ordered Bratteli diagram $B$ with a premorphism $f:B\rightarrow C$ so that the induced map $\mathcal V(f):(X_C,T_C)\rightarrow (Y_B, S_B)$ does not make  topological factoring. Nevertheless, having some ``structural" properties on the diagrams $B$ and $C$, the existence of  an ordered premorphism $f:B\rightarrow C$ will lead to a topological factoring. For instance,  if $C$ is simple,  $B$ is  of finite rank and the map $\alpha=\mathcal V(f)$ is finite-to-one on $\alpha^{-1}(X_B^{\rm max})$ then $\alpha$ is a topological factoring (Proposition \ref{simple}). It is likely  that if both $B$ and $C$ are simple and decisive then $\alpha=\mathcal V(f)$ is a topological factoring. 

Some equivalent conditions to the existence of a topological factoring induced by a premorphism between two semi-decisive Bratteli diagrams are provided (Proposition \ref{equiv}).
As a direct consequence of having topological factoring induced by a premorphism between two decisive Bratteli diagrams, we  show  that every  non-proper decisive Bratteli diagram of rank 2 is  conjugate to an odometer or it is a disjoint union of two odometers (Proposition \ref{odo}).

Theorem \ref{fact2}  provides a combinatorial model for topological factoring  between two zero dimensional dynamical systems using  ordered premorphisms. This can  be viewed as a generalization of the well-known Curtis–Hedlund–Lyndon theorem for modelling   factor maps between two zero-dimensional dynamical systems  through sequences of {\it sliding block codes}. To establish this result, we will prove the following theorem presented in the framework of S-adic representations of ordered Bratteli diagrams and ordered premorphisms as detailed in \cite{don, gh18}. See Section 5 for the notations used in this theorem.
 \begin{theorem}\label{inversefactor}
Consider zero dimensional dynamical systems $(X,T,{ X_0})$ and $(Y, S, { Y_0})$
where $X_0$ and $Y_0$ are quasi-sections. 
Then there exists
$ \pi:(X,T)\longrightarrow (Y,S)\ \ {\rm with} \ \  { \pi(X_0)\subseteq Y_0}$
if and only if  for every  K-R systems $\{\mathcal P_n\}_{n\geq 0}$ and $\{\mathcal Q_n\}_{n\geq 0}$ for $(Y,S,Y_0)$ and $(X,T,X_0)$, respectively, and the inverse limit systems associated to them,  there exists a sequence of natural numbers $\{n_i\}_{i\geq 0}$ and a sequence of sliding block codes $\pi_i:(\tilde{\mathcal Q}_{n_i},\sigma)\rightarrow (\tilde{\mathcal P}_i,\sigma)$ for all $i\geq 0$  such that 
 all the following rectangles between the inverse limit sequences commute:
 \begin{equation}\label{diag}
\xymatrix{(\tilde{Q}_{0},\sigma)\ar[d]_{\pi_{0}}
&(\tilde{Q}_{n_1},\sigma)\ar[l]_{_{\gamma_{1}}}\ar[d]_{\pi_{1}} &(\tilde{Q}_{n_2},\sigma)\ar[l]_{\gamma_{2}}\ar[d]_{\pi_{2}} &\cdots\ar[l]_{\gamma_{3}}& (X,T,X_0)\ar[l]\ar[d]_{\pi}\ \ \  \\
(\tilde{P}_{0},\sigma)
&(\tilde{P}_1,\sigma)\ar[l]^{\beta_1} &(\tilde{P}_{2},\sigma)\ar[l]^{\beta_2}&\cdots\ar[l]^{\beta_3}&(Y,S, Y_0)\ar[l] \   
 }
\end{equation}
where $\gamma_i:=\alpha_{n_i+1}\circ\alpha_{n_i+2}\circ\cdots\circ\alpha_{n_{i+1}}$.
\end{theorem}
In fact, in the above theorem, $(X,T,X_0)$ and $(Y,S,Y_0)$ were represented as the inverse limit systems of their (intermediate) symbolic factors, $(\tilde{Q}_{i},\sigma)$ and $(\tilde{P}_{i},\sigma)$ respectively. So for every $i\geq 1$, $\alpha_i$ and $\beta_i$ are the connecting maps between the intermediate factors. 
 When the system $(Y,S)$ is essentially minimal, then by  Theorem \ref{fact2} and Theorem \ref{main}, the argument of Theorem \ref{inversefactor} is somewhat ``simplified." Indeed, in this case, there exists a sequence of morphisms $\eta_i:\mathcal Q_{n_i}\rightarrow \mathcal P_i^*$ that guarantee  the existence of the sliding block codes $\pi_i:(\tilde{\mathcal Q}_{n_i},\sigma)\rightarrow (\tilde{\mathcal P}_i,\sigma)$ (see Proposition \ref{localmor}). 
Moreover, if $(X,T)$ and $(Y,S)$ are minimal  subshifts, then there exists some $i\geq 1$ such that $(X,T)\simeq (\tilde{\mathcal Q}_{n_i},\sigma)$ and $(Y,S)\simeq (\tilde{\mathcal P}_{i},\sigma)$ where $\simeq$ denotes conjugacy. Therefore, when we have a factor map $\pi:(X,T)\rightarrow (Y,S)$,   this is modelled by the sliding block code $\pi_i$ for a sufficiently large $i$.

The structure of this paper is as follows. In Section 2, we fix some notations and for the convenience of the reader, we first  recall some definitions and theorems from \cite{gps95, hps92, dk16, dk17}.  Then we investigate some of the basic properties of the induced map from an ordered premorphism that are used in the sequel such as a sufficient condition for having almost one-to-one extension induced by an ordered premorphism (Lemma \ref{kto1}).  Some necessary and/or sufficient conditions for having topological factoring out of an ordered premorphism between two semi-decisive Bratteli diagrams, are provided.

In Section 3, we describe a method for constructing a (semi-)decisive Bratteli diagram $B'$ for a given (semi-)decisive Bratteli diagram $B$ and  an ordered premorphism $f:B\rightarrow B'$ so that the map $\mathcal V(f)$  is not a topological factoring between $(X_B,T_B)$ and $(Y_{B'},S_{B'})$. This will lead to the proof of Theorem \ref{main}.

In Section 4,  we deal with  the realization of topological factorings between two zero dimensional systems $(X,T)$ and $(Y,S)$ by ordered premorphisms to prove Theorem~\ref{fact2}. We recall the notion of Kakutani-Rokhlin partitions for zero dimensional systems as was discussed in \cite{dk17, hps92, med, Poon}. Naturally, in this section, the ordered Bratteli diagrams constructed  are perfect in the sense of \cite{BY}, i.e., $(X,T)$ and $(Y,S)$ are realized by the Vershik maps on $X_B$ and $Y_C$ that are homeomorphisms. 

 In Section 5, we will prove the generalization of the well-known Curtis–Hedlund–Lyndon theorem for modelling   factor maps between two zero-dimensional dynamical systems  by sequences of  sliding block codes.

 \section{Preliminaries} 
 \subsection{Zero Dimensional  Dynamical Systems.}\label{subsec_top}
A {\it zero dimensional dynamical system} is a pair $(X, T)$ where $X$ is a non-empty compact totally disconnected metric space and $T$ is a homeomorphism on $X$. 
The orbit of a  point $x\in X$, denoted by ${\mathcal O}(x)$, is  the sequence $(T^nx)_{n\in\mathbb{Z}}$. If $X$ has finitely many points then $(X,T)$ is called a {\it trivial } dynamical system. If $X$ is a Cantor space (that is, a nonempty compact metrizable totally disconnected space with no isolated points)
then the system is called a {\it Cantor system}. Two topological dynamical systems $(X, T)$ and $(Y, S)$ are  {\it semi-conjugate} if there exists a surjective continuous map $\alpha: X\rightarrow Y$ such that $\alpha\circ T=S\circ \alpha$. In this case $(Y, S)$ is called a {\it factor} of $(X, T)$,  $(X, T)$ is called an \emph{extension} of $(Y, S)$,
and $\alpha$ is called a \emph{factor map} or a {\it topological factoring}.
\smallskip

For a topological dynamical system $(X, T)$ if the orbits of all points  are infinite, then the systems is called {\it aperiodic}. When all the orbits of the points are dense
in $X$, then the system is called {\it minimal}. This is equivalent to the absence of
non-trivial invariant closed subsets. When  $(X, T)$  has a unique minimal subsystem,
 the system is called {\it essentially minimal}  \cite{bs02, dk16}.

\subsection{Bratteli–Vershik models of zero dimensional systems}\label{versh}
In this subsection we recall some of the definitions related to ordered Bratteli diagrams from \cite{BY, dk16, hps92} and some of the main results of \cite{dk16, dk17}.

\begin{definition}
\hspace{-0.2cm}
\begin{itemize}
\item A 
\emph{Bratteli diagram}
$B=(V, E)$ 
consists of an infinite sequence of finite, non-empty, pairwise disjoint sets 
$V_{0} = \{ v_{0} \}, V_{1}, V_{2}, \ldots$, 
called the vertices, another sequence of finite, non-empty, pairwise disjoint sets 
$E_{1}, E_{2}, \ldots$, 
called the 
\emph{edges}, 
and two maps 
$s : E_{n} \rightarrow V_{n-1}, r : E_{n} \rightarrow V_{n}$, 
for every 
$n \geq 1$, 
called the range and source maps, such that 
$r^{-1}\{ v \}$ 
is non-empty for all 
$v$ 
in 
$\cup_{n \geq 1} V_{n}$ 
and 
$s^{-1}\{ v \}$ 
is non-empty for all 
$v$ 
in 
$\cup_{n \geq 0} V_{n}$. For every $n\geq 1$ we have an {\it adjacency}  matrix $M_n$ of size $|V_n|\times |V_{n-1}|$ that its entries $M_n^{ij}$ shows the number of edges between $v_i\in V_n$ and $v_j\in V_{n-1}$. 
\item A Bratteli diagram is called {\it  simple} if there exists some sequence $\{n_k\}_{k\geq 1}$ so that $$\forall k\geq 1\ \ M_{n_k}\cdot M_{n_k+1}\cdots M_{n_{k+1}}>0.$$

\item 
An 
\emph{ordered Bratteli diagram}
$B=(V, E, \leq)$
consists of a Bratteli diagram 
$(V, E)$ 
and a partial order
$\leq$
on 
$E$ 
such that two edges 
$e, e^{\prime}$ 
in 
$E$ 
are comparable if and only if 
$r(e) = r(e^{\prime})$. 
In such a diagram, we let 
$E_{\max}$ 
and 
$E_{\min}$ 
denote the set of maximal and minimal edges, respectively.
\end{itemize}
\end{definition}

 \begin{definition}[\cite{BY, dk16}]\label{order}
Let  $B=(V,E,\leq)$ 
be an ordered Bratteli diagram and $X_B$ be the the compact space of all (partially ordered) infinite paths of $B$. 
\begin{enumerate}
\item The Vershik map $T_B:X_B\setminus X_B^{\rm max} \rightarrow X_B\setminus X_B^{\rm min}$ is defined by
$$T_B(e_0, e_1, \ldots, e_\ell,e_{\ell+1},\ldots)=(0,0,\ldots, e_{\ell}+1,e_{\ell+1},\ldots)$$
where $\ell$ is the first index that $e_\ell$ is not the max edge in $r^{-1}(r(e_{\ell}))$ and  $0$ denotes the min edge in $r^{-1}(r(e_i))$  for every  $i\geq 0$. 
\item $B$ is called \emph{perfect} if the Vershik map $T_B:X_B\setminus X_B^{\rm max}\rightarrow X_B\setminus X_B^{\rm min}$ can be extended to a homeomorphism $T_B: X_B\rightarrow X_B$. 

\item $B$ is called
\emph{decisive} 
if the Vershik map extends in a unique way to a homeomorphism 
$\bar{T}_B$ 
of 
$X_{B}$. 
A zero dimensional dynamical system 
$(X, T)$ 
will be called \emph{Bratteli–Vershikizable} if it is conjugate to 
$(X_{B}, T_{B})$ 
for a decisive ordered Bratteli diagram 
$B$. So every decisive Bratteli diagram is perfect. 
\item $B$  is called \emph{properly ordered} if it has a unique infinite min path and a unique infinite max path. Clearly every properly ordered Bratteli diagram is decisive. 
\end{enumerate}
\end{definition}

We refer the reader to \cite{BY, dk16, dk17} to see various examples of perfect or decisive Bratteli diagrams.

\begin{lemma}[\cite{dk16}, Lemma~6.11]\label{lem81}
An ordered Bratteli diagram is decisive if and only if the following two conditions hold:
\begin{enumerate}
\item
the Vershik map and its inverse are uniformly continuous on their domains, and
\item
the set of maximal paths and the set of minimal paths either both have empty interiors, or both their interiors consist of just one isolated point.
\end{enumerate}
\end{lemma}

According to 
\cite[Proposition~1.2]{dk17}, 
the second condition in the previous lemma is equivalent to this condition: the domains of the Vershik map and its inverse are either both dense in 
$X_{B}$ 
or their closures both miss one point (not necessarily the same).

The following result gives a necessary and  sufficient condition for a system to be Bratteli–Vershikizable.

\begin{theorem}[\cite{dk17}, Theorem~3.1]\label{thm82}
A  zero dimensional system 
$(X, T)$ 
is Bratteli-Vershikizable if and only if the set of aperiodic points is dense, or its closure misses one periodic orbit.
\end{theorem}

\subsection{Ordered Premorphisms}
\begin{definition}[\cite{aeg21}, Definition~2.5]\label{def61}
Let 
$B=(V,E,\geq)$ 
and 
$C=(W,S,\geq)$ 
be ordered Bratteli diagrams. By an 
\emph{ordered premorphism} 
(or just a 
\emph{premorphism} 
if there is no confusion) 
$f: B\to C$ 
we mean a triple 
$(F, (f_{n})_{n=0}^{\infty},\geq )$ 
where
$(f_{n})_{n=0}^{\infty}$ 
is an  unbounded sequence of positive integers with
$f_{0}=0\leq f_{1}\leq f_{2}\leq\cdots$,
$F$ 
consists of a disjoint union
$F_{0}\cup F_{1}\cup F_{2}\cup\cdots$ 
together with a pair of range and source maps
$r:F\to W$, 
$s:F\to V$, 
and
$\geq$ 
is a partial order on 
$F$ 
such that:

\begin{enumerate}
\item\label{def61_1}
each 
$F_{n}$ 
is a non-empty finite set, 
$s(F_{n})\subseteq V_{n}$,
$r(F_{n})\subseteq W_{f_{n}}$,  $F_{0}$
is a singleton,
$s^{-1}\{v\}$ 
is non-empty for all 
$v$ 
in 
$V$, 
and
$r^{-1}\{w\}$ 
is non-empty for all 
$w$
in 
$W$;
\item\label{def61_2}
$e,e'\in F$ 
are comparable if and only if 
$r(e)=r(e')$, 
and 
$\geq$ 
is a linear order on 
$r^{-1}\{w\}$, 
for all 
$w\in W$;

\item\label{def61_3}
the diagram of 
$f: B \to C$,

\[
\xymatrix{V_{0}\ar[r]^{E_{1}}\ar[d]_{F_{0}}
&V_{1}\ar[r]^-{E_{2}}\ar[d]_{F_{1}} &V_{2}\ar[r]^-{E_{3}}\ar[d]_{F_{2}} &\cdots \ \ \  \\
W_{f_{0}}\ar[r]_{S_{f_{0},f_{1}}}
&W_{f_{1}}\ar[r]_{S_{f_{1},f_{2}}} &W_{f_{2}}\ar[r]_{S_{f_{2},f_{3}}}&\cdots \   ,
 }
\]
commutes. The ordered commutativity of the diagram of 
$f$  
means that for each
$n\geq 0$,
$E_{n+1}\circ F_{n+1}\cong F_{n}\circ S_{f_{n},f_{n+1}}$, 
i.e., there is a (necessarily unique) bijective map from 
$E_{n+1}\circ F_{n+1}$ 
to
$F_{n}\circ S_{f_{n},f_{n+1}}$
preserving the order and intertwining the respective source and range maps.
\end{enumerate}
\end{definition}
To see how  the ordered premorphism $f:B\to C$ induces a well-defined function  $\alpha:X_{C}\to X_{B}$ between the two Vershik systems, 
let $x=(s_{1},s_{2},\ldots)$ be  an infinite path in $X_{C}$.
Define the path $\alpha(x)=(e_{1},e_{2},\ldots)$ in $X_{B}$ as follows.
Fix $n\geq 1$.
By Definition~\ref{def61}, the  diagram

 \[
\xymatrix{V_{0}\ar[r]^{F_0}\ar[d]_{E_{0,n}}
 &W_0\ar[d]^{S_{0,f_n}} \\
 V_n\ar[r]_{F_n}
 &W_{f_{n}}
 }
\]
commutes,  that is, $F_{0}\circ S_{0,f_{n}}\cong E_{0,n}\circ F_{n}$. Thus, there is a unique path
$(e_{1},e_{2},\ldots,e_{n},d_{n})$ in $E_{0,n}\circ F_{n}$ (in fact $(e_1, e_2, \ldots, e_n)\in E_{0,n}$ and $d_n\in F_n$), corresponding to
the path $(d_{0},s_{1},\ldots,s_{f_{n}})$ in
$F_{0}\circ S_{0,f_{n}}$
where $d_{0}$ is the unique element of $F_{0}$. So
the path $\alpha(x)=(e_{1},e_{2},\ldots)$ in $X_{B}$
is associated to the path $x=(s_{1},s_{2},\ldots)$  in $X_{C}$. 

\begin{proposition}\label{factor1}
Let $B$ and $C$ be two ordered Bratteli diagrams and $f:B\rightarrow C$ be an ordered premorphism between them. Let $\alpha=\mathcal V(f):X_C\rightarrow X_B$ be its induced map. Consider the Vershik homeomorphisms $T_B: X_B\setminus X_B^{\rm max}\rightarrow X_B\setminus X_B^{\rm min}$ and $T_C:X_C\setminus X_C^{\rm max}\rightarrow X_C\setminus X_C^{\rm min}$. Then
\begin{enumerate}
\item \label{onto} $\alpha:X_C\rightarrow X_B$ is continuous and surjective.
\item\label{max} $\alpha(X_C^{\rm min})\subseteq X_B^{\rm min}$ and $\alpha(X_C^{\rm max})\subseteq X_B^{\rm max}$.
\item\label{intersect} $\alpha\circ T_C(x)=T_B\circ \alpha(x)$ for every $x\in \alpha^{-1}(X_B\setminus X_B^{\rm max})$. 
\end{enumerate}
\end{proposition}
\begin{proof}
First we show that 
$ \alpha $
is continuous. Let  
$ f = ( F, (f_{n})_{n=0}^{\infty}, \leq ) $
as in Definition~\ref{def61}.
By the definition of 
$ \alpha $
if
$ x=( x_{j})_{j=1}^{\infty}, y=(y_{j})_{j=1}^{\infty} \in X_{C} $
and
$ n\in N $
satisfy 
$ x_{j} = y_{j} $
for all
$ 1 \leq j \leq f_{n}, $
then 
$ \alpha(x)_{j} = \alpha(y)_{j} $
for all 
$ 1 \leq j \leq n $
where 
$ \alpha(x) = (\alpha(x)_{j})_{j=1}^{\infty} $
and 
$ \alpha(y) = (\alpha(y)_{j})_{j=1}^{\infty}. $
Thus
\[ 
\alpha (C(x_{1},\ldots,x_{f_{n}})) \subseteq C(\alpha(x)_{1},\ldots,\alpha(x)_{n})
\]
for all $ n \in \mathbb N$.
This shows that 
$ \alpha $
is continuous.
Now we show that 
$\alpha $
is surjective.
Let 
$ z = ( z_{1},z_{2},\ldots) $
be in 
$ X_{B}$, i.e.,
an infinite path in 
$ E.$
Fix 
$ n\geq 1 $.
By Definition ~\ref{def61},  the diagram 
 \[
\xymatrix{V_{0}\ar[r]^{F_0}\ar[d]_{E_{0,n}}
 &W_0\ar[d]^{S_{0,f_n}} \\
 V_n\ar[r]_{F_n}
 &W_{f_{n}}
 }
\]
is ordered commutative, that is, 
$ F_{0} \circ S_{0,f_{n}} \cong E_{0,n} \circ F_{n} . $
Thus, there is a unique path
$ (e_{0},s_{1},\ldots,s_{f_{n}}) \in F_{0,n} \circ S_{0,f_{n}} ,$
corresponding to the path 
$ (z_{1},\ldots,z_{n},e_{n}) \in E_{0,n} \circ F_{n} . $
Let 
$ x \in X_{C} $
be any infinite path such that 
$ x_{j} = s_{j} $
for 
$ 1 \leq j \leq f_{n}. $
Then 
$ \alpha(x) \in C(z_{1},\ldots,z_{n}) $
and 
$ \alpha^{-1}(C(z_{1},\ldots,z_{n})) \neq \varnothing .$
We have: 
$$ \alpha^{-1}(C(z_{1})) \supseteq \alpha^{-1}(C(z_{1},z_{2})) \supseteq \cdots  $$
also each 
$ \alpha^{-1}(C(z_{1},\ldots,z_{n})) $
is compact, since 
$ C(z_{1},\ldots,z_{n}) $
is closed and 
$ \alpha $
is continuous.
Then by compactness of
$ X_{C} $
we have 
$$ \bigcap _ {n=1} ^ \infty \alpha^{-1}(C(z_{1},\ldots,z_{n})) \neq \varnothing .$$
Take any 
$ x \in \bigcap _ {n=1} ^ \infty \alpha^{-1}(C(z_{1},\ldots,z_{n})),$
then
$ \alpha (x) = z $
and 
$\alpha $
is surjective. Note that 
$ \alpha ^ {-1} ({z}) = \bigcap _ {n=1} ^ \infty \alpha^{-1}(C(z_{1},\ldots,z_{n})) .$

Parts (2) and (3) follow from the definition of $\alpha$.
 \end{proof}
 
\begin{definition}[\cite{aeg21}, Definition 2.10]\label{iso1}
Let $f,g:B\rightarrow C$ be two ordered premorphisms with $B=(V,E,\leq)$, $C=(W,E',\leq)$, $f=(F, (f_n)_{n\geq 0}, \leq)$, and $g=(G, (g_n)_{n\geq 0},\leq)$. It is said that $f$ is equivalent to $g$, $f\sim g$,  if for each $n\geq 0$ there is an $m\geq f_n, g_n$ such that $F_n\circ S_{f_n,m}\cong G_n\circ S_{g_n,m}$ (order isomorphism as in  part (3) of Definition~\ref{def61}). 
\end{definition}

Now we recall the notion of {\it order isomorphism} between two finite {\it ordered sets} which was previously considered in the proof of \cite[Proposition 5.2]{gh18}, and  we will need it to verify isomorphism between two ordered premorphisms.

Let $V$ and $W$ be two finite non-empty sets. We say that $F$ is an {\it ordered set of edges from $V$ to $W$} if $F$ is a finite non-empty set with a partial ordering $\leq$ on it and with surjective source and range maps, $s_F:F\rightarrow V$ and $r_F:F\rightarrow W$, such that $e, e'\in F$ are comparable if and only if $r_F(e)=r_F(e')$, and the restriction of $\leq$ to each set $r_F^{-1}(w)$, $w\in W$, is a total ordering. We use the notation $F:V\rightarrow W$. 

If $G:V\rightarrow W$ is another ordered set of edges from $V$ to $W$, then we say $F$ is order isomorphisc to $G$, $F\cong G$,  if there is a (necessarily unique) bijective map from $F$ to $G$ preserving the range and the source maps.

 If $S:W\rightarrow U$ is another ordered set of edges, then the composition  of $F$ and $S$ is 
$$F\circ S=\{(t,g)\in F\times S: \ r(t)=s(g)\}$$
endowed with the reverse lexicographisc order. Then $F\circ S:V\rightarrow U$ is an ordered set of  edges. 

\begin{lemma}\label{compat}
Let $F_1:V_1\rightarrow W_1$ , $F_2, G_2: V_2\rightarrow W_2$, $E:V_1\rightarrow  V_2$, and $S:W_1\rightarrow W_2$ be ordered sets of edges and consider the following two order commutative diagrams (one by $F_2$ and the other one by $G_2$):
\begin{center}
\begin{tikzpicture}[scale=1.2]
\node at (7,11) {$V_1$};
\node at (9,11) {$W_1$};
\node at (7,9) {$V_2$};
\node at (9.2,9) {$W_2.$};
\draw[->, thick] (7.2,9.2) [out=30, in=150] to (8.85,9.15);
\draw[->, thick] (7.2,8.8) [out=-30, in=210] to (8.85,8.88);
\draw[->, thick] (7,10.8) [out=-90, in=90] to (7,9.2);
\draw[->, thick] (9,10.8) [out=-90, in=90] to (9,9.2);
\draw[->, thick] (7.2,11) to (8.7,11);
\node at (8,11.3) {$F_1$};
\node at (8,9.7) {$F_2$};
\node at (8,8.3) {$G_2$};
\node at (6.75,10) {$E$};
\node at (9.2,10) {$S$};
\end{tikzpicture}
\end{center}
Suppose that $E\circ F_2\cong F_1\circ S\cong E\circ G_2$ with  first coordinate compatible order isomorphisms (i.e., if $\gamma:E\circ F_2\rightarrow F_1\circ S$ and $\eta:E\circ G_2\rightarrow F_1\circ S$ are the order isomorphisms, then $\gamma(e,f)=\eta(e',g)$ implies that $e=e'$, for all $(e,f)\in E\circ F_2$ and $(e',g)\in E\circ G_2$). Then $F_2\cong G_2$.
\end{lemma}
\begin{proof}
To prove $F_2\cong G_2$, it is enough to show that if $w\in W_2$ and $$r_{F_2}^{-1}(w)=\{f_1, \ldots, f_n\}, \ \ r_{G_2}^{-1}(w)=\{g_1,\ldots, g_m\}$$
are the two totally ordered sets associated to $F_2$ and $G_2$ respectively, then $m=n$ and $s_{F_2}(f_i)=s_{G_2}(g_i)$ for all $1\leq i\leq n$. We first show that the latter is true for $i=1$. Let $e_1, e'_1\in E$ be the minimal edges in $E$ with $r_E(e_1)=
s_{F_2}(f_1)$ and $r_E(e'_1)=s_{G_2}(g_1)$. Then $(e_1, f_1)$ (resp., $(e'_1,g_1)$) is the minimal path in $E\circ F_2$ (resp. $E\circ G_2$) with range $w$. Since $E\circ F_2\cong E\circ G_2$ as ordered sets, we get $\gamma(e_1,f_1)=\eta(e'_1, g_1)$. Applying the assumption, we see that $e_1=e'_1$. Thus $s_{F_2}(f_1)=r_E(e_1)=s_{G_2}(g_1)$.  Let $e_2, e'_2\in E$ be the minimal edges in $E$ with $r_E(e_2)=s_{F_2}(f_2)$ and $r_E(e'_2)=s_{G_2}(g_2)$. Let  $k=\#r_E^{-1}(r_E(e_1))$. Then $(e_2, f_2)$ (resp., $(e'_2, g_2)$) is the $(k+1)$-th path in $E\circ F_2$ (resp. $E\circ G_2$) with range $w$.  Again $E\circ F_2\cong E\circ G_2$ implies that $\gamma(e_2, f_2)=\eta(e'_2, g_2)$ and hence $e_2=e'_2$. In particular, $s_{F_2}(f_2)=r_E(e_2)=s_{G_2}(g_2)$. This also shows that $n\geq 2$ iff $m\geq 2$. Continuing this procedure, we get $n=m$ and $s_{F_2}(f_i)=s_{G_2}(g_i)$ for all $1\leq i\leq n$. 
\end{proof}
\begin{proposition} \label{unique1} 
 Let $B$ and $C$ be two ordered Bratteli diagrams. Suppose that $f,g:B\rightarrow C$ are two ordered premorphism between them. The surjective continuous induced maps $\mathcal V(f), \mathcal V(g):X_C\rightarrow X_B$ are the same if and only if $f$ and $g$ are equivalent in the sense of \cite[Definitions 2.8 and 2.10]{aeg21}.
\end{proposition}
\begin{proof}
Let $B=(V, E,\leq)$, $C=(W, S,\leq)$, and  $f=(F, (f_n)_{n\geq 0}, \leq)$, $g=(G, (g_n)_{n\geq 0}, \leq)$. Suppose that $\mathcal V(f)=\mathcal V(g)$. Fix $n\in\mathbb N$. Without loss of generality, assume that $f_n\leq g_n$. We want to show that $F_n\circ S_{f_n, g_n}\cong G_n$ or in other words, the following diagram commutes which implies that $f$ is equivalent to $g$ in the sense of \cite[Definition 2.10]{aeg21}:
\begin{center}
\begin{tikzpicture}[scale=1.2]
\node at (7,11) {$V_n$};
\node at (9.5,11) {$W_{f_n}$};
\node at (9.5,9) {$W_{g_n}.$};
\draw[->, thick] (7.2,11) to (9.1,11);
\draw[->, thick] (7.1,10.8) to (9.1, 9.2);
\draw[->, thick] (9.4,10.8) to (9.4, 9.25);
\node at (8.25,11.3) {$F_n$};
\node at (9.9, 10) {$S_{f_n,g_n}$};
\node at (7.7,10) {$G_n$};
\end{tikzpicture}
\end{center}
To show this, consider the following diagram:
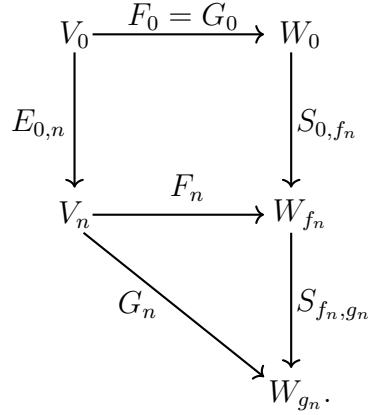
\begin{center}
\begin{tikzpicture}[scale=1.2]
\node at (7,13) {$V_0$};
\node at (9.5,13) {$W_{0}$};
\node at (7,11) {$V_n$};
\node at (9.5,11) {$W_{f_n}$};
\node at (9.5,9) {$W_{g_n}.$};
\draw[->, thick] (7.2,13) to (9.1,13);
\draw[->, thick] (7.2,11) to (9.1,11);
\draw[->, thick] (7.1,10.8) to (9.1, 9.2);
\draw[->, thick] (9.4,10.8) to (9.4, 9.3);
\draw[->, thick] (9.4,12.8) to (9.4, 11.3);
\draw[->, thick] (7,12.8) to (7, 11.3);
\node at (8.25,11.3) {$F_n$};
\node at (9.8, 12) {$S_{0, f_n}$};
\node at (6.6, 12) {$E_{0, n}$};
\node at (8.2,13.2) {$F_0=G_0$};
\node at (9.87, 10) {$S_{f_n,g_n}$};
\node at (7.7,10) {$G_n$};
\end{tikzpicture}
\end{center}
Since $f$ and $g$ are ordered premorphisms, we have
$$E_{0,n}\circ G_n\cong G_0\circ S_{0,g_n}\cong F_0\circ S_{0,f_n}\circ S_{f_n,g_n}\cong E_{0,n}\circ F_n\circ S_{f_n,g_n}.$$
Put $F'_n=F_n\circ S_{f_n, g_n}$ as an ordered set of edges from $V_n$ to $W_{g_n}$. Then we get the following diagram:
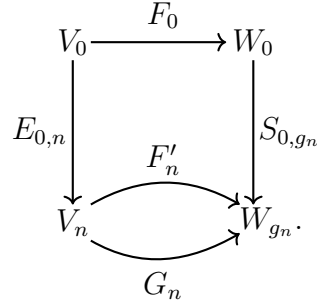
\begin{center}
\begin{tikzpicture}[scale=1.2]
\node at (7,11) {$V_0$};
\node at (9,11) {$W_0$};
\node at (7,9) {$V_n$};
\node at (9.2,9) {$W_{g_n}.$};
\draw[->, thick] (7.2,9.2) [out=30, in=150] to (8.85,9.2);
\draw[->, thick] (7.2,8.8) [out=-30, in=210] to (8.85,8.88);
\draw[->, thick] (7,10.8) [out=-90, in=90] to (7,9.2);
\draw[->, thick] (9,10.8) [out=-90, in=90] to (9,9.2);
\draw[->, thick] (7.2,11) to (8.7,11);
\node at (8,11.3) {$F_0$};
\node at (8,9.7) {$F'_n$};
\node at (8,8.3) {$G_n$};
\node at (6.63,10) {$E_{0,n}$};
\node at (9.4,10) {$S_{0,g_n}$};
\end{tikzpicture}
\end{center}
The isomorphisms $E_{0,n}\circ F'_n\cong F_0\circ S_{0,g_n}\cong E_{0,n}\circ G_n$ are compatible with respect to the first coordinate. Because, if $e, e'\in E_{0,n}$, $h\in G_n$, and $h'\in F'_n$ and both paths $(e,h)\in E_{0,n}\circ G_n$ and $(e', h')\in E_{0,n}\circ F'_n$ correspond to the same path $t\in F_0\circ S_{0,g_n}$, then $e=e'$. In fact, if we write $t=s_0x_1\cdots x_{g_n}$ where $F_0=\{s_0\}$ and $x_i\in S_i$ for $1\leq i\leq g_n$, then we can find $x_i\in S_i$ for $i>g_n$ such that $x=x_1x_2\cdots\in X_C$. By the definition of $\mathcal V(g)$, we see that $e=(\mathcal V(g)(x))_{[1,n]}$, i.e., the path $e$ is the initial part of the infinite path $\mathcal V(g)(x)\in X_B$. On the other hand, there is $p\in F_n$ such that $e'p\in E_{0,n}\circ F_n$ corresponds to the path $s_0x_1\cdots x_{f_n}\in F_0\circ S_{0,f_n}$ under the ordered isomorphism $E_{0,n}\circ F_n\cong F_0\circ S_{0,f_n}$. Hence $e'px_{f_n+1}\cdots x_{g_n}$ corresponds to $s_0x_1\cdots x_{g_n}$ under the isomorphism $E_{0,n}\circ F_n\circ S_{f_n, g_n}\cong F_0\circ S_{0,g_n}$. Thus $h'=px_{f_n+1}\cdots x_{g_n}$. By the definition of $\mathcal V(f)(x)$, we see that $e'=(\mathcal V(f)(x))_{[1,n]}$. Since $\mathcal V(f)=\mathcal V(g)$, we have that $e=e'$. Applying Lemma~\ref{compat}, it follows that $G_n\cong F'_n=F_n\circ S_{f_n,g_n}$ as was desired. Therefore, $f$ is equivalent to $g$. 

For the other direction, suppose that $f$ is equivalent to $g$. Let $x=x_1x_2\cdots\in X_C$ and $n\in\mathbb N$. We show that
 $$(\mathcal V(f)(x))_{[1,n]}=(\mathcal V(g)(x))_{[1,n]}.$$
 Without loss of generality, assume that $f_n\leq g_n$. By \cite[Definition 2.10]{aeg21}, there is $k\geq g_n$ such that $F_n\circ S_{f_n,k}\cong G_n\circ S_{g_n,k}$ as in the following diagram:
 \begin{center}
\begin{tikzpicture}[scale=1.2]
\node at (7,13) {$V_0$};
\node at (9.5,13) {$W_{0}$};
\node at (7,11) {$V_n$};
\node at (9.5,11) {$W_{f_n}$};
\node at (9.5,9) {$W_{g_n}$};
\node at (9.5,7) {$W_k.$};
\draw[->, thick] (7.2,13) to (9.1,13);
\draw[->, thick] (7.2,11) to (9.1,11);
\draw[->, thick] (7.1,10.8) to (9.1, 9.2);
\draw[->, thick] (9.4,10.8) to (9.4, 9.3);
\draw[->, thick] (9.4,12.8) to (9.4, 11.3);
\draw[->, thick] (7,12.8) to (7, 11.3);
\draw[->, thick] (9.4,8.8) to (9.4, 7.3);
\node at (8.25,11.3) {$F_n$};
\node at (9.8, 12) {$S_{0, f_n}$};
\node at (6.6, 12) {$E_{0, n}$};
\node at (8.2,13.2) {$F_0=G_0$};
\node at (9.87, 10) {$S_{f_n,g_n}$};
\node at (7.7,10) {$G_n$};
\node at (9.8, 8) {$S_{g_n,k}$};
\end{tikzpicture}
\end{center}
Since $F_0\circ S_{0,f_n}\cong E_{0,n}\circ F_n$, there is $e\in E_{0,n}$ and $p\in F_n$ such that the path $ep\in E_{0,n}\circ F_n$ corresponds to $s_0x_1\cdots x_{f_n}\in F_0\circ S_{0,f_n}$. In particular, $r(p)=r(x_{f_n})$. Similarly, there is $e' \in E_{0,n}$ and $q\in G_n$ such that $e'q\in E_{0,n}\circ G_n$ corresponds to $s_0x_1\cdots x_{g_n}\in G_0\circ S_{0,g_n}$. (We assume that $G_0=F_0=\{s_0\}$.) On the other hand, since $F_n\circ S_{f_n,k}\cong G_n\circ S_{g_n,k}$, the path $px_{f_{n}+1}\cdots x_k\in F_n\circ S_{f_n,k}$ corresponds to $q'q''\in G_n\circ S_{g_n,k}$ for some $q'\in G_n$ and $q''\in S_{g_n,k}$. Therefore, both paths $eq'q''$ and $e'qx_{g_n+1}\cdots x_k$ in $E_{0,n}\circ G_n\circ S_{g_n,k}$ correspond to the same path $s_0x_1\cdots x_k\in F_0\circ S_{0,k}$ under the isomorphism $E_{0,n}\circ G_n\circ S_{g_n,k}\cong F_0\circ S_{0,k}$. This implies that  $eq'q''=e'qx_{g_n}\cdots x_k$ and so $e=e'$. Consequently, 
$$(\mathcal V(f)(x))_{[1,n]}=e=e'=(\mathcal V(g)(x))_{[1,n]} $$
and the proof is finished.
\end{proof}
\begin{remark}
The proof of the first part of Proposition~\ref{unique1} may simplify    Definition~\ref{iso1} in the following way: two ordered premorphisms $f,g:B\rightarrow C$ are equivalent if and only if for every $n\geq 0$, $F_n\circ S_{f_n,m}\cong G_n\circ S_{g_n,m}$ where $m={\rm max}(f_n,g_n)$. 
\end{remark}
 \begin{lemma}\label{kto1}
Let $B$ and $C$ be two ordered Bratteli diagrams and $f:B\rightarrow C$ be an ordered premorphism  as in Definition \ref{def61}. Let $\alpha=\mathcal V(f):X_C\rightarrow X_B$ be its induced map. If $y=y_1y_2\cdots\in X_B$ and 
 $$\exists K\geq 1\ \forall n\geq 1 \ \ \#s_f^{-1}(r(y_n))\leq K$$
 (where $s_f$ denotes the source map of the ordered premorphism $f$) then $\#\alpha^{-1}(y)\leq K$. In particular, if for every vertex $v$ on $B$, $s_f^{-1}(v)$ is a singleton then $\alpha$ is one-to-one (hence it is a homeomorphism).
 \end{lemma}
 \begin{proof}
 Suppose  to the contrary that $\#\alpha^{-1}(y)=\{x^{(1)},x^{(2)},\ldots, x^{(m)}\}$, where $m>K$. Being distinct points in $X_C$,  after an appropriate telescoping, $x^{(i)}$'s can be realized as infinite paths on $C$ such that for some sufficiently large $\ell$: $$x^{(i)}_{[1,\ell]}\neq x^{(j)}_{[1,\ell]}\ {\rm  for }\ 0\leq i, j\leq m$$ 
 where $x^{(i)}_{[1,\ell]}$ is the initial finite path of length $\ell$ of the point $x^{(i)}$.
Then by definition of $\alpha$ and  Definition \ref{def61}\eqref{def61_3}, 
$$ \exists f_1, f_2,\ldots, f_m\in F_\ell \  \ s_f(f_i)=r(y_\ell),\ r_f(f_i)=r(x^{(i)}_{[1,\ell]}),\  i=1,\ldots, m.$$
See Figure \ref{K} as an example. In particular, $f_1, f_2, \ldots, f_m$ are distinct. This contradicts the assumption. \qedhere
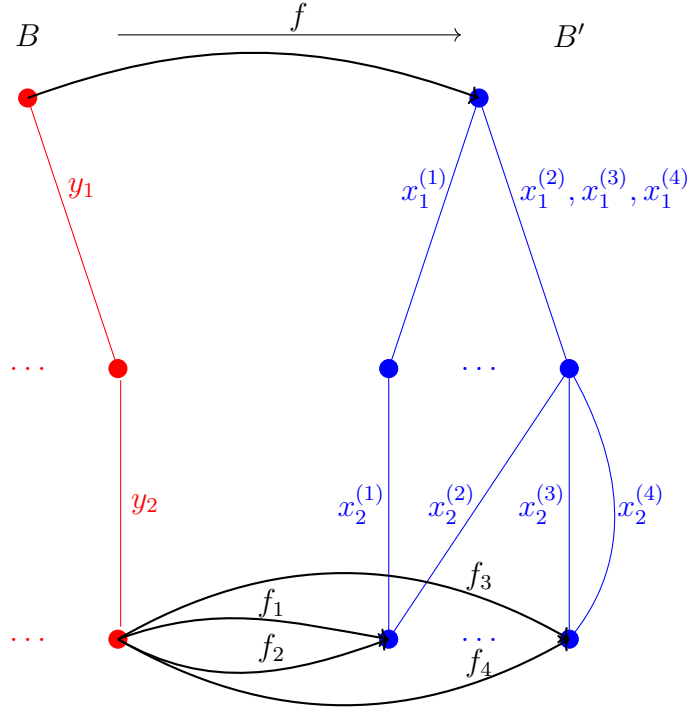
\begin{figure}
\begin{center}
\begin{tikzpicture}[scale=1.2]
{\color{blue}
\filldraw (11,10) circle [radius=0.1];
 \filldraw (10,7) circle [radius=0.1];
\filldraw (12,7) circle [radius=0.1];
 \filldraw (10,4) circle [radius=0.1];
\filldraw (12,4) circle [radius=0.1];


\draw (11,10)--(10,7);

\draw (11,10)--(12,7);

\draw (10,7)--(10,4);

\draw (12,7)--(12,4);
\draw (12,7)--(10,4);
\draw (12,7) [out=300, in=50] to (12,4);

}

{\color{red}
\filldraw (6,10) circle [radius=0.1];
\filldraw (7,7) circle [radius=0.1];
\filldraw (7,4) circle [radius=0.1];


\draw (6,10)--(7,7);
\draw (7.03,6.87)--(7.03,4.14);

}

\draw[->, thick] (6,10) [out=20, in=160] to (11,10);

\draw[->, thick] (7,4) [out=-30, in=200] to (10,4);
\draw[->, thick] (7,4) [out=20, in=170] to (10,4);
\draw[->, thick] (7,4) [out=30, in=150] to (12,4);
\draw[->, thick] (7,4) [out=330, in=210] to (12,4);


\node at (6,10.7) {$B$};
\node at (12,10.7) {$B'$};
\node at (9,10.9) {$f$};
\draw[->] (7,10.7) to (10.8,10.7);
\node at (6,7) {${\color{red} \ldots}$};
\node at (6,4) {${\color{red} \ldots}$};
\node at (11,7) {${\color{blue} \ldots}$};
\node at (11,4) {${\color{blue} \ldots}$};
\node at (6.6,9) {${\color{red} y_1}$};
\node at (7.3,5.5) {${\color{red} y_2}$};

\node at (10.4,9) {${\color{blue} x^{(1)}_1}$};
\node at (9.7,5.5) {${\color{blue} x^{(1)}_2}$};

\node at (12.4,9) {${\color{blue} x^{(2)}_1,x^{(3)}_1,x^{(4)}_1}$};

\node at (10.7,5.5) {${\color{blue} x^{(2)}_2}$};

\node at (11.7,5.5) {${\color{blue} x^{(3)}_2}$};

\node at (12.8,5.5) {${\color{blue} x^{(4)}_2}$};

\node at (8.7,4.4) {$f_1$};
\node at (8.7,3.9) {$f_2$};

\node at (11,4.7) {$f_3$};
\node at (11,3.73) {$f_4$};

\end{tikzpicture}
\end{center}
\caption{ An example illustrating the proof of Lemma~\ref{kto1}. Here $\ell=2$, $m=4$.}\label{K}
\end{figure} \end{proof}

 \begin{remark}
 In Lemma \ref{kto1}, the sufficient condition for  injectivity of $\alpha$, is equivalent to saying that for every $n$, the number of finite paths from $W_0$ to $W_{f_n}$ on diagram $C$ is equal to the number of finite paths from $V_0$ to $V_n$ on $B$. Moreover, when  $B$ and $C$ are decisive, $B$ is simple and $\alpha=\mathcal V(f):X_C\rightarrow X_B$ is a topological factoring, if there exists one infinite path $y=y_1y_2\cdots$ such that 
 $$ \forall n\geq 1\ \  \ \#s_f^{-1}(r(y_n))=1,$$
 then $(X_C,T_C)$ is an almost 1-1 extension of $(X_B,T_B)$. 
  \end{remark}

The following proposition can be proved by using ordered premorphisms. 
\begin{proposition}\label{odo}
If $B$ is a non-proper decisive ordered Bratteli diagram of rank $2$ then $(X_B, T_B)$ is conjugate to an odometer or it is a disjoint union of two odometers. 
\end{proposition}
\begin{proof}
We first enumerate vertexes of each level $V_i$ by $\{v^i_0, v^i_1\}$ from left to right. Decisiveness implies that there are two infinite max paths  $y^{(k)}=y^{(k)}_1y^{(k)}_2\cdots$, $k=0,1,$ and two infinite min paths $x^{(k)}=x^{(k)}_1x^{(k)}_2\cdots$, $k=0,1,$ on $B$.  We can assume  that (in fact, after an appropriate telescoping) for each $k=0,1$, the infinite max (resp. min) path $y^{(k)}$ (resp. $x^{(k)}$) is carried by $v^i_k$, at every  level $i$. There are two possibilities for $B$:
\begin{enumerate}
\item $T_B(y^{(k)})=x^{(k)}$ for every $k=0,1$. Then  continuity of $T_B$  implies that after finitely many levels, there is not any cross edges between any two consecutive  levels of the diagram, i.e., for sufficiently large $i$,
$$\nexists e\in E_i \ \ s(e)=v^{i-1}_k, \ r(e)=v^{i}_{1-k}.$$
In other words, there is $i_0\geq 0$ such that for each infinite path ${\bf e}=e_1e_2\cdots$ on $B$ there exists $k\in\{0,1\}$ such that  for every $i\geq i_0$, $s(e_i)=v^{i-1}_k, r(e_i)=v^i_k$. 
This turns out to have two odometers, each one supported by a single vertex at each level. The blue diagram on the right side of   Figure \ref{cantorfig} (which is in fact related to Example~\ref{cantor})  is an example of this case.
\item $T_B(y^{(k)})=x^{(1-k)}$. This time (after a telescoping of the diagram along some cofinal sequence) continuity of $T_B$ forces existences of cross edges between every two levels of $B$ in a way that for every $i\geq 1$ if $$E^{(k)}_i=\{e_0=e_{\rm min},e_1,\dots,e_{\rm max}\}$$  is the set of edges ranged at $v^i_k\in V_i$, $k=0,1,$ then 
$$\forall j=0, 1,\ldots, {\rm max} \ \ s(e_{j})=v^{i-1}_{k+j\ ({\rm mod} \ 2)}.$$
See the red diagrams on the left sides of Figure \ref{rank2}  and Figure \ref{cantorfig}.
Then one can create   an ordered premorphism $f:B\rightarrow C$ where $C=(W, E')$ is a rank one Bratteli diagram (and so  its Vershik system is clearly an odometer) such that the map $\alpha=\mathcal V(f):X_C\rightarrow X_B$ makes a topological factoring between the two systems. See the example in Figure \ref{rank2}. The diagram $C$ has the property that for each $i\geq 1$ the number of finite paths from $W_0$ to $W_i$ is equal to the total number of paths from $V_0$ to $V_i=\{v^i_0, v^i_1\}$. So by Lemma \ref{kto1}, $\alpha$ is a homeomorphism. Consequently, we have a conjugacy between the two systems.  \qedhere
\end{enumerate}
\end{proof}

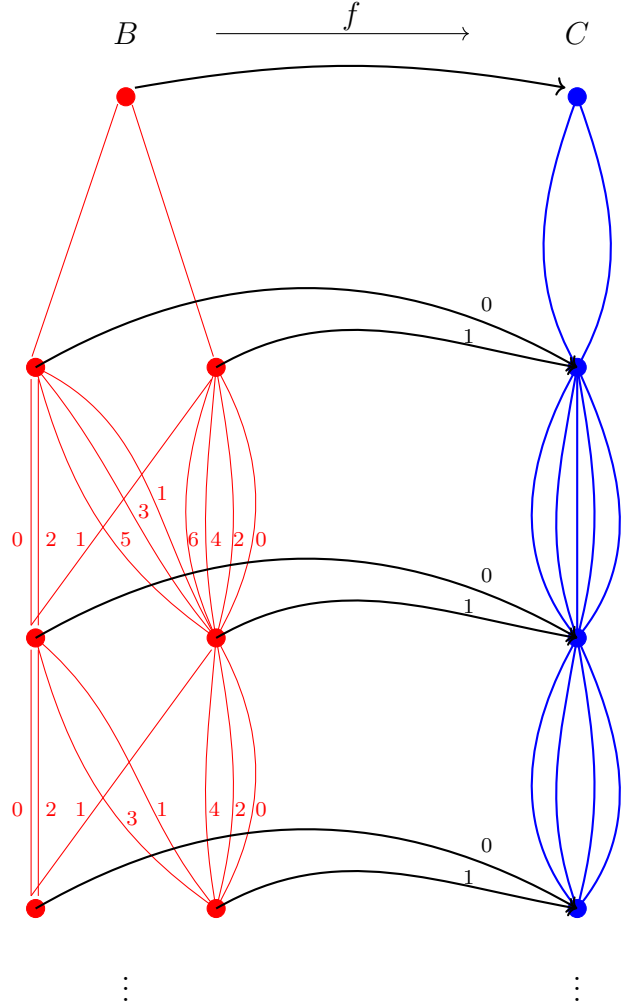
\begin{figure}
\begin{center}
\begin{tikzpicture}[scale=1.2]
{\color{blue} 
\filldraw (11,10) circle [radius=0.1];
 \filldraw (11,7) circle [radius=0.1];
\filldraw (11,4) circle [radius=0.1];
 \filldraw (11,1) circle [radius=0.1];

\draw[-, thick] (11,10) [out=250, in=120] to (11,7);
\draw[-, thick] (11,10) [out=290, in=60] to (11,7);

\draw[-, thick] (11,7) [out=240, in=130] to (11,4);
\draw[-, thick] (11,7) [out=295, in=50] to (11,4);
\draw[-, thick] (11,7) [out=260, in=110] to (11,4);
\draw[-, thick] (11,7) [out=280, in=75] to (11,4);
\draw[-, thick] (11,7) to (11,4);

\draw[-, thick] (11,4) [out=240, in=130] to (11,1);
\draw[-, thick] (11,4) [out=295, in=50] to (11,1);
\draw[-, thick] (11,4) [out=260, in=110] to (11,1);
\draw[-, thick] (11,4) [out=280, in=75] to (11,1);

}
{\color{red}
\filldraw (6,10) circle [radius=0.1];
 \filldraw (5,7) circle [radius=0.1];
\filldraw (7,7) circle [radius=0.1];
 \filldraw (5,4) circle [radius=0.1];
\filldraw (7,4) circle [radius=0.1];
 \filldraw (5,1) circle [radius=0.1];
\filldraw (7,1) circle [radius=0.1];

\draw (5.91,9.92)--(4.96,7.12);
\draw (6.07,9.91)--(6.97,7.12);

\draw (4.95,6.87)--(4.95,4.14);
\draw (4.95,3.87)--(4.95,1.14);
\draw (5.03,6.87)--(5.03,4.14);
\draw (5.03,3.87)--(5.03,1.14);

\draw (6.95,6.87)--(4.95,4.14);
\draw (6.95,3.87)--(4.95,1.14);

\draw (7,7) [out=300, in=60] to (7,4);
\draw (7,7) [out=280, in=75] to (7,4);
\draw (7,7) [out=265, in=100] to (7,4);
\draw (7,7) [out=250, in=115] to (7,4);
\draw (5,7) [out=285, in=145] to (7,4);
\draw (5,7) [out=310, in=130] to (7,4);
\draw (5,7) [out=335, in=120] to (7,4);

\draw (5,4) [out=285, in=145] to (7,1);
\draw (5,4) [out=320, in=130] to (7,1);
\draw (7,4) [out=300, in=60] to (7,1);
\draw (7,4) [out=280, in=75] to (7,1);
\draw (7,4) [out=265, in=100] to (7,1);


\node at (6.4,5.6) {\tiny{1}};
\node at (6.2,5.4) {\tiny{3}};
\node at (6,5.1) {\tiny{5}};
\node at (7.25,5.1) {\tiny{2}};
\node at (7,5.1) {\tiny{4}};
\node at (6.75,5.1) {\tiny{6}};
\node at (7.5,5.1) {\tiny{0}};


\node at (5.5,5.1) {\tiny{1}};
\node at (5.17,5.1) {\tiny{2}};
\node at (4.8,5.1) {\tiny{0}};


\node at (5.5,2.1) {\tiny{1}};
\node at (5.17,2.1) {\tiny{2}};
\node at (4.8,2.1) {\tiny{0}};

\node at (7.5,2.1) {\tiny{0}};
\node at (7.27,2.1) {\tiny{2}};
\node at (6.4,2.1) {\tiny{1}};
\node at (6.07,2) {\tiny{3}};
\node at (6.98,2.1) {\tiny{4}};

}

\draw[->, thick] (6.1,10.1) [out=10, in=170] to (10.87,10.1);

\draw[->, thick] (5,7) [out=30, in=150] to (11,7);
\draw[->, thick] (7,7) [out=30, in=170] to (11,7);

\draw[->, thick] (5,4) [out=30, in=150] to (11,4);
\draw[->, thick] (7,4) [out=30, in=170] to (11,4);

\draw[->, thick] (5,1) [out=30, in=150] to (11,1);
\draw[->, thick] (7,1) [out=30, in=170] to (11,1);


\node at (10,7.7) {\tiny{0}};
\node at (9.8,7.35) {\tiny{1}};

\node at (10,4.7) {\tiny{0}};
\node at (9.8,4.35) {\tiny{1}};

\node at (10,1.7) {\tiny{0}};
\node at (9.8,1.35) {\tiny{1}};

\node[very thick] at (6,0.2) {\vdots};
\node[very thick] at (11,0.2) {\vdots};
\node at (6,10.7) {$B$};
\node at (11,10.7) {$C$};
\node at (8.5,10.9) {$f$};
\draw[->] (7,10.7) to (9.8,10.7);
\end{tikzpicture}
\end{center}
\caption{An example illustrating the proof of Proposition~\ref{odo}. Here $\mathcal V(f): X_C\rightarrow X_B$ is a conjugacy. }\label{rank2}
\end{figure}

\bigskip

We consider a weaker version of decisiveness, called {\it semi-decisive}, to obtain more general results in studying topological factoring between two ordered Bratteli diagrams.

\begin{definition}\label{semi}
We say an ordered Bratteli diagram $B$ is  {\it semi-decisive} if the Vershik map $T_B:X_B\setminus X_B^{\rm max}\rightarrow X_B\setminus X_B^{\rm min}$ has a continuous surjective extension $\bar{T}_B:X_B\rightarrow X_B$.
\end{definition}

\begin{figure}
\begin{center}
\begin{tikzpicture}[scale=1.2]
{\color{blue}
\filldraw (11,10) circle [radius=0.1];
 \filldraw (10,7) circle [radius=0.1];
\filldraw (12,7) circle [radius=0.1];
 \filldraw (10,4) circle [radius=0.1];
\filldraw (12,4) circle [radius=0.1];
 \filldraw (10,1) circle [radius=0.1];
\filldraw (12,1) circle [radius=0.1];

\draw (10.91,9.92)--(9.96,7.12);

\draw (11.07,9.91)--(11.97,7.12);
\draw (11.17,9.91)--(12.06,7.13);

\draw (9.95,6.87)--(9.95,4.14);
\draw (9.95,3.87)--(9.95,1.14);
\draw (10.03,6.87)--(10.03,4.14);
\draw (10.03,3.87)--(10.03,1.14);

\draw (11.95,6.87)--(11.95,4.14);
\draw (11.95,3.87)--(11.95,1.14);
\draw (12.03,6.87)--(12.03,4.14);
\draw (12.03,3.87)--(12.03,1.14);

}

{\color{red}
\filldraw (6,10) circle [radius=0.1];
 \filldraw (5,7) circle [radius=0.1];
\filldraw (7,7) circle [radius=0.1];
 \filldraw (5,4) circle [radius=0.1];
\filldraw (7,4) circle [radius=0.1];
 \filldraw (5,1) circle [radius=0.1];
\filldraw (7,1) circle [radius=0.1];

\draw (5.91,9.92)--(4.96,7.12);

\draw (6.07,9.91)--(6.97,7.12);

\draw (4.95,6.87)--(4.95,4.14);
\draw (4.95,3.87)--(4.95,1.14);
\draw (5.03,6.87)--(5.03,4.14);
\draw (5.03,3.87)--(5.03,1.14);

\draw (6.95,6.87)--(4.95,4.14);
\draw (6.95,3.87)--(4.95,1.14);
\draw (7.03,6.87)--(7.03,4.14);
\draw (7.03,3.87)--(7.03,1.14);




\node at (5.5,5.1) {\tiny{1}};
\node at (5.17,5.1) {\tiny{2}};
\node at (4.8,5.1) {\tiny{0}};


\node at (5.5,2.1) {\tiny{1}};
\node at (5.17,2.1) {\tiny{2}};
\node at (4.8,2.1) {\tiny{0}};

}
\draw[->, thick] (6.1,10.1) [out=10, in=170] to (10.87,10.1);


\draw[->, thick] (5.15,7.05) [out=10, in=170] to (11.78,7.1);
\draw[->, thick] (7.15,6.95) [out=-10, in=190] to (9.85,6.88);

\draw[->, thick] (7.15,6.95) [out=40, in=156] to (11.85,7.24);



\draw[->, thick] (5.15,4.05) [out=10, in=170] to (11.78,4.1);
\draw[->, thick] (7.15,3.95) [out=-30, in=200] to (9.86,3.88);
\draw[->, thick] (7.15,3.95) [out=20, in=170] to (9.86,3.98);
\draw[->, thick] (7.15,3.95) [out=40, in=156] to (11.78,4.24);


\draw[->, thick] (5.15,1.05) [out=10, in=170] to (11.78,1.1);
\draw[->, thick] (7.15,0.95) [out=-55, in=240] to (9.85,0.86);
\draw[->, thick] (7.15,0.95) [out=-25, in=210] to (9.85,0.88);
\draw[->, thick] (7.15,0.95) [out=-10, in=190] to (9.85,0.95);
\draw[->, thick] (7.15,0.95) [out=20, in=160] to (9.85,1);

\draw[->, thick] (7.15,0.95) [out=40, in=156] to (11.78,1.24);


\node at (11.1,7.3) {\tiny{0}};
\node at (11.3,7.6) {\tiny{1}};

\node at (11.1,4.3) {\tiny{0}};
\node at (11.3,4.6) {\tiny{1}};

\node at (11.1,1.3) {\tiny{0}};
\node at (11.3,1.6) {\tiny{1}};

\node[very thick] at (6,0.2) {\vdots};
\node[very thick] at (11,0.2) {\vdots};
\node at (6,10.7) {$B$};
\node at (11.2,10.7) {$C$};
\node at (8.7,11.1) {$f$};
\draw[->] (6.5,10.7) to (10.8,10.7);
\end{tikzpicture}
\end{center}
\caption{ Related to Example \ref{cantor}. The induced map $\mathcal V(f):X_C\rightarrow X_B$ is not a topological factoring while $f:B\rightarrow C$ is an ordered premorphism.  Note that by the proof of Proposition~\ref{odo} the left diagram is conjugate to 2-odometer. }\label{cantorfig}
\end{figure}
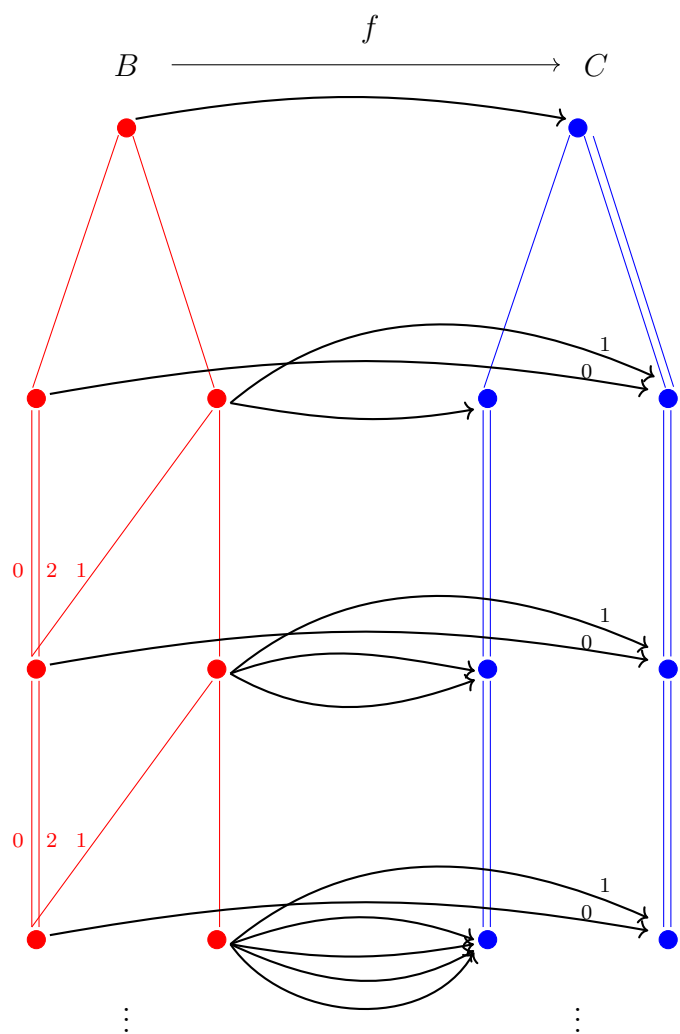

\begin{remark}
Every  decisive ordered Bratteli diagram is semi-decisive.
\end{remark}
\begin{proposition}\label{equiv}
Let $B$ and $C$ be two semi-decisive ordered Bratteli diagrams. Consider Vershik maps $\bar{T}_B:X_B\rightarrow X_B$ and $\bar{T}_C:X_C\rightarrow X_C$ such that $\bar{T}_B(X_B^{\rm max})\subseteq X_B^{\rm min}$ and  $\bar{T}_C(X_C^{\rm max})\subseteq X_C^{\rm min}$. Suppose that $f:B\rightarrow C$ is an ordered premorphism as in Definition \ref{def61}  and  $\alpha=\mathcal V(f):X_C\rightarrow X_B$ is the induced contiuous surjective map. Then $\alpha$ is a topological factoring if and only if for any $y=(e_1, e_2,\ldots)\in X_B^{\rm max}$ and  every $x=(s_1, s_2,\ldots)\in\alpha^{-1}(y)$ one of the following occurs:
\begin{enumerate} 
\item $x\in X_C^{\rm max}$ and for every $n\in\N$ there is a minimal edge $p_n\in F_n$ such that $s(p_n)=r(e'_n)$ and $r(p_n)=r(s'_{f_n})$ where $e'_n$ is the $n$-th edge of $\bar{T}_B(y)$ and $s'_{f_n}$ is the $f_n$-th edge of $\bar{T}_C(x)$.
\item $x\notin X_C^{\rm max}$ and if $k$ is the smallest number with $r(s_{f_k})=r(s'_{f_k})$ then there exist $p_1, p_2,\ldots$ in $F$ such that for any $n<k$, $p_n$ is a minimal edge in $F_n$, $s(p_n)=r(e'_n)$ and $r(p_n)=r(s'_{f_n})$, and for all $n\geq k$, $p_n$ is the successor of $d_n$ where $d_n$'s are those edges in $F$ realizing $\alpha(x)=y$ (as in the paragraph preceding \cite[Lemma 3.13]{aeg21}).
\end{enumerate}
\end{proposition}
\begin{proof}
First note that by Proposition \ref{factor1}\eqref{intersect}, $\alpha\circ T_C(x)=T_B\circ \alpha(x)$ for every $x\in \alpha^{-1}(X_B\setminus X_B^{\rm max})$. For all other points  $x\in \alpha^{-1}(X_B^{\rm max})$ (by Proposition \ref{factor1}\eqref{max}), items (1) and (2) are precisely translations of the equation $\alpha(\bar{T}_C(x))=\bar{T}_B(y)$ in terms of diagrams and ordered  premorphism $f$. 
\end{proof}

\begin{remark}
The condition of having $\bar{T}_B(X_B^{\rm max})\subseteq X_B^{\rm min}$ occurs for considerable class of Vershik systems on ordered Bratteli diagrams including decisive Bratteli diagrams. In fact, if $B$ is semi-decisive and $X_B$ has no isolated points (hence, is a Cantor set) then $$\bar{T}_B(X_B^{\rm max}\cap (X_B\setminus X_B^{\max})')\subseteq X_B^{\rm min},$$
where $A'$ denotes the set of limit points of $A$.
 In addition, if ${\rm rank}(B)<\infty$ then $\bar{T}_B(X_B^{\rm max})\subseteq X_B^{\rm min}$ which follows from the fact that  $\#X_B^{\rm max}\leq {\rm rank}(B)$. In the sequel, we will call $\bar{T}_B$ a {\it natural extension}  if it satisfies 
the condition $$\bar{T}_B(X_B^{\rm max})\subseteq X_B^{\rm min}.$$
 \end{remark}
Not every extension of a Vershik map, from $X_B\setminus X_B^{\rm max}$ to $X_B$, on a semi-decisive Bratteli diagram is  natural. For example, on the diagram of  \cite[Example 6.15]{dk16} one may define $\bar{T}_B$ on $X_B$ by mapping all the max paths (passing through the vertex $w$) to an arbitrarily chosen non-min path.

\begin{corollary}\label{seq1}
Let $B$ and $C$ be two semi-decisive ordered Bratteli diagrams such that there exist natural extensions $\bar{T}_B$ and $\bar{T}_C$ on them. Suppose that $f:B\rightarrow C$ is an ordered premorphism and  $\alpha=\mathcal V(f):X_C\rightarrow X_B$ is the induced continuous surjective map. Suppose that for every point $x\in \alpha^{-1}(X_B^{\rm max})$ there exists a sequence $(x_n)_{n\geq 1}$ in $ X_C\setminus X_C^{\rm max}$ such that $\lim_{n} x_n=x$ and $\alpha(x_n)\in X_B\setminus X_B^{\rm max}$ for all $n\in\mathbb N$. Then $\alpha$ is a topological factoring. 
\end{corollary}
\begin{proof}
Let $x\in X_C$. If $x\notin\alpha^{-1}(X_B^{\rm max})$ then $x\in X_C\setminus X_C^{\rm max}$, by Proposition~\ref{factor1}\eqref{max}. Hence $\alpha\circ T_C(x)=T_B\circ\alpha(x)$, by Proposition~\ref{factor1}\eqref{intersect}. Now suppose that $x\in\alpha^{-1}(X_B^{\rm max})$. By assumption there is a sequence $(x_n)_{n\geq 1}$ in $X_C$ such that $x_n\rightarrow x$ and $\alpha(x_n)\in X_B\setminus X_B^{\rm max}$. Thus $x_n\notin\alpha^{-1}(X_B^{\rm max})$ and hence $\alpha\circ T_C(x_n)=T_B\circ \alpha(x_n)$ for all $n\in\mathbb N$. Letting $n\rightarrow\infty$ and using continuity of $\alpha, T_C$, and $T_B$, we get $\alpha T_C(x)=T_B\alpha(x)$. Therefore, $\alpha$ is a topological factoring.
\end{proof}

\begin{example}\label{cantor}
In Figure~\ref{cantorfig},  both diagrams $B$ and $C$ are decisive and non-proper. The Vershik map of the left diagram is minimal. The  map $f:B\rightarrow C$ is an ordered premorphism but its induced map $\alpha$ is not  a topological factoring. In fact, $X_C=X_{C_1}\cup X_{C_2}$ and for every $x\in X_{C_1}$ (the left wing) the factoring equation fails. Note that both diagrams are Cantor sets. Moreover, the preimages of the left max path of the left diagram $B$ under the map $\alpha$ are contained in $X_C\setminus X_C^{\rm max}$. In other words, there is an infinite max path on the left diagram with no max path preimage. 
\end{example}

\begin{proposition}\label{simple}
Let $B$ and $C$ be two semi-decisive  ordered Bratteli diagrams such that  there exist natural extensions $\bar{T}_B$ and $\bar{T}_C$ on $X_B$ and $X_C$ respectively. Suppose that  $f:B\rightarrow C$ is an ordered premorphism with the induced continuous surjective map $\alpha=\mathcal V(f):X_C\rightarrow X_B$. If $X_C$ has no isolated points (in particular, if $C$ is simple and $X_C$ is non-trivial)  and ${\#\alpha^{-1}(X_B^{\rm max})}<\infty$ (in particular, if $B$ is of finite rank and $\alpha$ is finite-to-one) then $\alpha$ is a topological factoring.
\end{proposition}
\begin{proof}
We use Corollary~\ref{seq1}.   Let $x\in \alpha^{-1}(X_B^{\rm max})$. Since $X_C$ has no isolated points and $\alpha^{-1}(X_B^{\rm max})$ is finite, there exists a sequence  $(x_n)_{n\geq 1}$ of distinct points in $X_C\setminus \alpha^{-1}(X_B^{\rm max})$ such that $x_n\rightarrow x$. 
Thus  $\alpha(x_n)\in X_B\setminus X_B^{\rm max}$ for all $n\geq 1$. Now Corollary~\ref{seq1} implies that $\alpha$ is a topological factoring.
\end{proof}

\begin{Question}
Let $B$ and $C$ be two (semi-)decisive simple (hence minimal) ordered Bratteli diagrams such that  there exist natural extensions  $\bar{T}_B$ and $\bar{T}_C$ on $X_B$ and $X_C$ respectively. Suppose that $f:B\rightarrow C$ is an ordered premorphism and  $\alpha=\mathcal V(f):X_C\rightarrow X_B$ is the induced continuous surjective map. Is the map $\alpha$  a topological factoring?
\end{Question}

\section{Ordered Premorphisms Versus Topological Factoring}\label{sec8}

\subsection{Construction.}\label{construction}
Throughout this section, unless otherwise stated explicitly, we assume that  $B=(V,\, E,\, \geq)$ is an ordered Bratteli diagram, $z\in X_B^{\rm min}$ and $y\in X_B^{\rm max}$. We construct an ordered Bratteli diagram $B'=(V',\,E'\,\geq')$ and an ordered premorphism $f:B\rightarrow B'$ such that there is $x\in X_{B'}$ which is not an infinite max path and the induced continuous map $\alpha=\mathcal V(f):X_{B'}\rightarrow X_B$ satisfies $$\alpha(x)=y\ \ \ \ {\rm and} \ \ \ \alpha(T_{B'}x)=z.$$
\smallskip
Let $V=\cup_{n=0}^\infty V_n$, $E=\cup_{n=1}^\infty E_n$, $k_n:=\#V_n,$ and $V_n=\{v_1^n,\dots, v_{k_n}^n\}$ for $n\geq 1$ and $V_0=\{v_0\}$. We are going to add a vertex $v_0^n$ at each level $V_n$, $n\geq 1$. So
$$V'_0=V_0,\ \ V'_n=\{v_0^n\}\cup\{V_n\}, n\geq 1.$$
To define $E'$ using $z$ and $y$ suppose that $z=z_1z_2\cdots$ and $y=y_1y_2\cdots$ where $z_n, y_n\in E_n$ for $n\geq 1$.
Let $r^{-1}(r(z_1))=\{g_{1,1},\dots, g_{1,\ell_1}\}$ and $r^{-1}(r(y_1))=\{h_{1,1},\dots, h_{1,m_1}\}$ where 
$g_{1,1}< g_{1,2}<\cdots< g_{1,\ell_1}$ and $h_{1,1}< h_{1,2}<\cdots< h_{1,m_1}$. 
Thus $g_{1,1}=z_1$ and $h_{1,m_1}=y_1$. 
For each $g_{1,j}$ and $h_{1,i}$ we consider the edges $g'_{1,j}$ and $h'_{1,i}$ on $E'_1$ with the range $v_0^1$ and similar sources as $g_{1,j}$ and $h_{1,i}$. So
$$E'_1=r^{-1}(v^1_0)\cup E_1,$$
$$r^{-1}(v^1_0)=\{h'_{1,1},\dots, h'_{1,m_1}, g'_{1,1}, \dots, g'_{1,\ell_1}\}. $$
The ordering on $r^{-1}(v^1_0)$ is as written above.
Now let $n\geq 2$ and define $E'_n$ as follows. Let 
$r^{-1}(r(z_n))=\{g_{n,1},\dots, g_{n,\ell_n}\}$ and $r^{-1}(r(y_n))=\{h_{n,1},\dots, h_{n,m_n}\}$ as ordered sets. Thus $g_{n,1}=z_n$ and $h_{n,m_n}=y_n$. As before, for each $g_{n,j}$ and $ h_{n,i}$ we consider distinct edges $g'_{n,j}$ and $h'_{n,i}$ for $E'_n$ except for $i=m_n$ and $j=n$ that we identify $h'_{n,m_n}$ and $g'_{n,1}$. In other words,
$$E'_n=r^{-1}(v_0^n)\cup E_n,$$
$$r^{-1}(v_0^n)= \{h'_{n,1}, \ldots, h'_{n,m_n}=g'_{n,1},  g'_{n,2}, \ldots, g'_{n,\ell_n}\}.$$
The ordering on $r^{-1}(v^n_0)$ is as written above.
The exception we made for $h'_{n,m_n}$ makes the following distinction for the source of this edge. In fact, for every 
$1\leq i< m_n$ and $1< j\leq \ell_n$  we have $s(h'_{n,i})=s(h_{n,i})$ and $s(g'_{n,j})=s(g_{n,j})$ but $s(h'_{n,m_n})=v_0^{n-1}$. So $h'_{n,m_n}$ is the only edge in $E'_n$ with the source $v_0^{n-1}$ for $n\geq 2$. 
Now consider the infinite path $x=x_1x_2\cdots$ where $x_n=h'_{n,m_n}$ for $n\geq 1$. Having $x_1=h'_{1,m_1}< g'_{1,1},$ one can conclude that $x$ is not an infinite max path on $X_{B'}$ and $T_{B'}(x)=g'_{1,1}x_2x_3\cdots$.

\smallskip
Let us define an ordered premorphism $f:B\rightarrow B'$ as follows. Set $f=(F, (f_n)_{n=0}^\infty,\geq)$ where $f_n=n$ for all $n\geq 0$ and $F=\cup_{n=0}^\infty F_n$ so that
$F_0$ has only one edge, $(0, v_0)$, with the source $v_0\in V_0$ and with the range $v_0\in V'_0$ and for every $n\geq 1$, $F_n$ is an ordered set from $V_n$ to $V'_n$ where $\# F_n=\#V_n+2$ such that for every $v\in V_n$ there is exactly one edge in $F_n$, say $(n,v)$, from $v\in V_n$ to $v\in V'_n$ and there are two other edges $s_n$ and $s'_n$ in $F_n$ such that 
$$s(s_n)=r(z_n), \ s(s'_n)=r(y_n), \ r(s_n)=r(s'_n)=v_0^n\ \ {\rm and }\ \ s'_n< s_n$$ 
as in Figure~\ref{sn}.
\begin{figure}
\begin{center}
\begin{tikzpicture}[scale=1.1]
{\color{blue}

\filldraw (10,4) circle [radius=0.1];
\node at (10.32,4.0) {\tiny{$v_0^n$}};

}
{\color{red}
\filldraw (5,6) circle [radius=0.1];
\node at (5,3.7) {\tiny{$r(z_n)$}};

\filldraw (7,6) circle [radius=0.1];
 \filldraw (5,4) circle [radius=0.1];
\filldraw (7,4) circle [radius=0.1];
\node at (7,3.7) {\tiny{$r(y_n)$}};

\draw (4.95,6.0)--(4.95,4.14);
\node at (7.2,5) {\tiny{$y_n$}};
\draw (7.03,6.0)--(7.03,4.14);
\node at (4.8,5) {\tiny{$z_n$}};
}
\draw[->, thick] (5.15,4.05) [out=10, in=170] to (9.85,4.1);
\draw[->, thick] (7.15,3.95) [out=-10, in=190] to (9.85,3.88);

\node at (9.6,4.3) {\tiny{1}};
\node at (9.6,3.99) {\tiny{0}};
\node at (8.6,4.4) {\tiny{$s_n$}};
\node at (8.6,3.96) {\tiny{$s'_n$}};

\end{tikzpicture}
\end{center}
\caption{$s_n$ and $s'_n$ are the edges in the premorphism with range $v_0^n$.}\label{sn}
\label{fig1}
\end{figure}
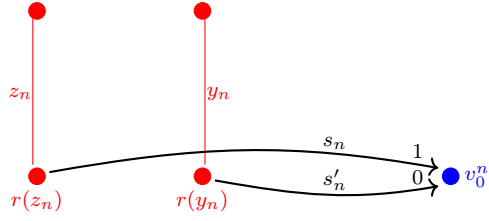
Thus 
$$F_n=\{(n,v_j^n):\ 1\leq j\leq k_n\}\cup\{s_n,s'_n\}$$
and
$$s((n,v_j^n))=v_j^n\in V_n,\  r((n,v_j^n))=v_j^n\in V'_n.$$

As there is only one edge in $F_n$ with range $v_j^n$ for every $1\leq j\leq k_n$, the ordering on $r^{-1}(v_j^n)=\{(n,v_j^n)\}$ is trivial.

\begin{figure}
\begin{center}
\begin{tikzpicture}[scale=1.2]
{\color{blue}
\filldraw (12,10) circle [radius=0.1];
\filldraw (10,7) circle [radius=0.1];
\filldraw (12,7) circle [radius=0.1];
\filldraw (14,7) circle [radius=0.1];
\filldraw (10,4) circle [radius=0.1];
\filldraw (12,4) circle [radius=0.1];
\filldraw (14,4) circle [radius=0.1];
\filldraw (10,1) circle [radius=0.1];
\filldraw (12,1) circle [radius=0.1];
\filldraw (14,1) circle [radius=0.1];

\draw (11.65,9.90) [out=190, in=85]  to (9.9,7.12);
\draw (11.68,9.85) [out=210, in=70]  to (9.93,7.05);
\draw (11.71,9.80) [out=240, in=50]  to (9.96,6.98);
\draw (11.73,9.72) [out=265, in=40]  to (9.99,6.91);

\draw (12,9.68)--(12,7.12);
\draw (12.08,9.68)--(12.08,7.12);

\draw (12.35,9.75)--(14.0,7.05);
\draw (12.41,9.80)--(14.05,7.1);

\draw (10,6.87)--(10,4.14);
\draw (12,6.87)--(12,4.14);
\draw (13.98,6.87)--(13.98,4.14);
\draw (14.06,6.87)--(14.06,4.14);
\draw (11.97,6.85)--(10.08,4.1);
\draw (13.95,6.85)--(10.18,4.1);
\draw (12.08,6.85)--(12.09,4.14);

\draw (10,3.87)--(10,1.14);
\draw (12,3.87)--(12,1.14);
\draw (13.98,3.87)--(13.98,1.14);
\draw (14.06,3.87)--(14.06,1.14);
\draw (11.97,3.85)--(10.08,1.1);
\draw (13.95,3.85)--(10.18,1.1);
\draw (12.08,3.85)--(12.09,1.14);


\node at (11.93,7.35) {\tiny{0}};
\node at (12.15,7.35) {\tiny{1}};

\node at (13.70,7.33) {\tiny{0}};
\node at (14.05,7.33) {\tiny{1}};

\node at (10.2,9) {\tiny{$0$}};
\node at (10.7,9) {\tiny{$1$}};
\node at (11.1,8.9) {\tiny{$2$}};
\node at (11.5,8.9) {\tiny{$3$}};

\node at (10.35,8.4) {\tiny{$x_1$}};
\node at (11,8.4) {\tiny{$(T_{B'}x)_1$}};

\node at (9.92,4.35) {\tiny{1}};
\node at (10.17,4.37) {\tiny{2}};
\node at (10.4,4.35) {\tiny{0}};

\node at (11.93,4.35) {\tiny{0}};
\node at (12.15,4.35) {\tiny{1}};

\node at (13.91,4.38) {\tiny{0}};
\node at (14.15,4.38) {\tiny{1}};

\node at (10,5.7) {\tiny{$x_2, (T_{B'}x)_2$}};

\node at (9.92,1.35) {\tiny{1}};
\node at (10.17,1.37) {\tiny{2}};
\node at (10.4,1.35) {\tiny{0}};

\node at (11.93,1.35) {\tiny{0}};
\node at (12.15,1.35) {\tiny{1}};

\node at (13.91,1.38) {\tiny{0}};
\node at (14.15,1.38) {\tiny{1}};

\node at (10,2.7) {\tiny{$x_3, (T_{B'}x)_3$}};

}
{\color{red}
\filldraw (6,10) circle [radius=0.1];
 \filldraw (5,7) circle [radius=0.1];
\filldraw (7,7) circle [radius=0.1];
 \filldraw (5,4) circle [radius=0.1];
\filldraw (7,4) circle [radius=0.1];
 \filldraw (5,1) circle [radius=0.1];
\filldraw (7,1) circle [radius=0.1];

\draw (5.91,9.92)--(4.96,7.12);
\draw (5.98,9.88)--(5.05,7.13);

\draw (6.07,9.91)--(6.97,7.12);
\draw (6.17,9.91)--(7.06,7.13);

\draw (4.95,6.87)--(4.95,4.14);
\draw (4.95,3.87)--(4.95,1.14);
\draw (5.03,6.87)--(5.03,4.14);
\draw (5.03,3.87)--(5.03,1.14);

\draw (6.95,6.87)--(6.95,4.14);
\draw (6.95,3.87)--(6.95,1.14);
\draw (7.03,6.87)--(7.03,4.14);
\draw (7.03,3.87)--(7.03,1.14);


\node at (5.06,7.7) {\tiny{0}};
\node at (5.32,7.7) {\tiny{1}};
\node at (6.75,7.5) {\tiny{0}};
\node at (7.05,7.5) {\tiny{1}};

\node at (6.7,8.7) {\tiny{$y_1$}};
\node at (5.38,8.7) {\tiny{$z_1$}};

\node at (6.89,4.39) {\tiny{0}};
\node at (7.1,4.39) {\tiny{1}};
\node at (5.12,4.39) {\tiny{1}};
\node at (4.89,4.39) {\tiny{0}};

\node at (7.16,5.7) {\tiny{$y_2$}};
\node at (4.84,5.7) {\tiny{$z_2$}};

\node at (6.89,1.39) {\tiny{0}};
\node at (7.1,1.39) {\tiny{1}};
\node at (5.12,1.39) {\tiny{1}};
\node at (4.89,1.39) {\tiny{0}};

\node at (7.16,2.7) {\tiny{$y_3$}};
\node at (4.84,2.7) {\tiny{$z_3$}};
}
\draw[->, thick] (6.1,10.1) [out=10, in=170] to (11.9,10.1);


\draw[->, thick] (5.15,7.05) [out=10, in=170] to (9.85,7.1);
\draw[->, thick] (7.15,6.95) [out=-10, in=190] to (9.85,6.88);

\draw[->, thick] (5.15,7.16) [out=20, in=156] to (11.85,7.24);

\draw[->, thick] (7.15,6.8) [out=-30, in=200] to (13.66,6.9);


\draw[->, thick] (5.15,4.05) [out=10, in=170] to (9.85,4.1);
\draw[->, thick] (7.15,3.95) [out=-10, in=190] to (9.85,3.88);

\draw[->, thick] (5.15,4.16) [out=20, in=156] to (11.78,4.24);

\draw[->, thick] (7.15,3.8) [out=-30, in=200] to (13.66,3.9);


\draw[->, thick] (5.15,1.05) [out=10, in=170] to (9.85,1.1);
\draw[->, thick] (7.15,0.95) [out=-10, in=190] to (9.85,0.88);

\draw[->, thick] (5.15,1.16) [out=20, in=156] to (11.78,1.24);

\draw[->, thick] (7.15,0.8) [out=-30, in=200] to (13.66,0.9);


\node at (9,6.95) {\tiny{0}};
\node at (9,7.4) {\tiny{1}};

\node at (9,3.95) {\tiny{0}};
\node at (9,4.4) {\tiny{1}};

\node at (9,1.4) {\tiny{1}};
\node at (9,0.9) {\tiny{0}};

\node[very thick] at (6,0.2) {{\color{red}\vdots}};
\node[very thick] at (12,0.2) {{\color{blue}\vdots}};
\node at (6,10.7) {$B$};
\node at (12,10.7) {$B'$};
\node at (9,10.9) {$f$};
\draw[->] (7,10.7) to (10.8,10.7);
\end{tikzpicture}
\end{center}
\caption{An example of the construction in Subsection~\ref{construction}}\label{ex}
\end{figure}
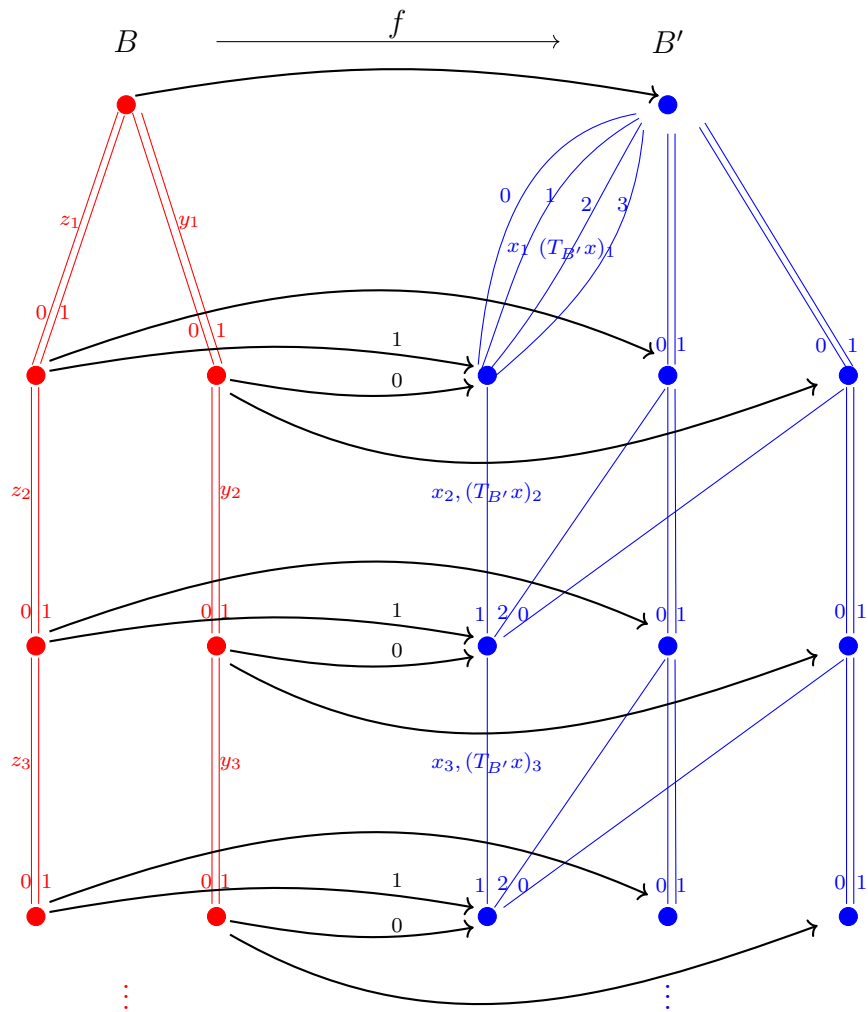

\begin{remark}
It is worth mentioning that  in the sequel, when we apply  the above method to construct an ordered Bratteli diagram $B'$ for a given ordered Bratteli diagram $B$, we choose $y\in X_B^{\rm max}  $ and $z\in X_B^{\rm min}$ such that $\bar{T}_B(y)\neq z$.
\end{remark}

\begin{lemma}
Let $B$ be an ordered Bratteli diagram and $B'$ and $f:B\rightarrow B'$ be as above. Then $f$ is an ordered premorphism.
\end{lemma}
\begin{proof} We only need to check the ordered commutativity of $f$. Let $n\geq 1$ and let $v\in V'_n$. Recall that $V'_n=\{v_0^n\}\cup V_n$. First let $v\neq v_0^n$. Suppose that $r^{-1}(v)=\{e_1,\dots, e_k\}\subseteq E_n$ as an ordered set. Then the set of paths in $E_n\circ F_n$ with range $v\in V'_n$ is $$\{(e_1,(n,v)), (e_2,(n,v)), \dots, (e_n,(n,v))\}$$ as an ordered set where $(n,v)$ is the only edge of $F_n$ going from $v\in V_n$ to $v\in V_n$. On the other hand, the set of paths in $F_{n-1}\circ E'_n$ with range $v\in V'_n$ is
$$\big\{((n-1, s(e_1)), e_1), ((n-1, s(e_2)), e_2), \dots, ((n-1, s(e_k)), e_k)\big\}$$
as an ordered set. Thus these two sets are ordered isomorphic. Now let $v=v_0^n$. Then the ordered set of paths in $E_n\circ F_n$ with range $v_0^n$ is 
$$\big\{(h_{n,1},s'_n),(h_{n,2},s'_n),\dots, (h_{n,m_n},s'_n), (g_{n,1},s_n), (g_{n,2},s_n),\dots, (g_{n,\ell_n},s_n)\big\}.$$
Moreover, the ordered set of paths in $F_{n-1}\circ E'_n$ with range $v_0^n$ is 
\begin{eqnarray*}
\big\{((n-1,s(h_{n,1})),h'_{n,1}),\dots, ((n-1,s(h_{n,m_n-1})),h'_{n,m_n-1}),\\
 (s'_{n-1},h'_{n,m_n}), (s_{n-1}, h'_{n,m_n}), ((n-1, s(g'_{n,2})),g'_{n,2}),\dots, ((n-1, s(g'_{n,\ell_n})),g'_{n,\ell_n})\big\}
\end{eqnarray*}
for $n\geq 2$. These two sets are ordered isomorphic. For $n=1$ the number of paths in $E_1\circ F_1$ with range $v^1_0$ and those in $F_0\circ E'_1$ with range $v^1_0$ are both $\ell_1+m_1$. Consequently, 
$$E_n\circ F_n\cong F_{n-1}\circ E'_n,\ \forall n\geq 1$$ as ordered sets.
\end{proof}

We summarize the construction of $B'$ and $f$ and their properties as follows.

\begin{theorem}\label{base}
Let $B$ be an ordered Bratteli diagram, $z\in X_B^{\rm min}$ and $y\in X_B^{\rm max}$. Consider  the associated  ordered Bratteli diagram $B'$ to $B$  with an infinite path $x\in X_{B'}$ together with the ordered premorphism $f:B\rightarrow B'$ as  in the previous lemma. Then for the induced continuous surjection $\alpha=\mathcal V(f):X_{B'}\rightarrow X_B$ the following assertions hold:
\begin{enumerate}
\item \label{nofactor}  $x\in X_{B'}\setminus X_{B'}^{\rm max}$,
$$\alpha(x)=y\ \ {\rm and}\ \ \alpha(T_{B'}x)=z.$$
\item $X_B\subseteq X_{B'}$ and $\alpha \upharpoonright_{X_B}={\rm id}$.
\item $B'$ is not simple and $f$ is stationary at all levels. (See Figure~\ref{ex} as an example.)
\item $x\in X_{B'}^{\rm min}$ if and only if $y\in X_B^{\rm min}$. Moreover, $T_{B'}x\in X_{B'}^{\rm max}$ if and only if $z\in X_B^{\rm max}$. 
\item\label{isolated}  Every point  in $X_{B'}\setminus X_B$ is isolated in $X_{B'}$. Moreover, if $w$ is an isolated point of $X_{B}$ then it is isolated in $X_{B'}$.
\end{enumerate}

\end{theorem}

\subsection{From Ordered Premorphism to Topological Factoring}

\begin{proposition}\label{factor3}
Suppose that $B$ and $C$ are two semi-decisive ordered Bratteli diagrams such that $X_B$ has a unique infinite min path. Then for every ordered premorphism $f:B\rightarrow C$ the induced map $\alpha=\mathcal V(f):X_C\rightarrow X_B$ is a topological factoring, that is $\alpha\circ \bar{T}_C=\bar{T}_B\circ \alpha$. 
\end{proposition}
\begin{proof}
 We show that $\alpha\circ \bar{T}_C(x)=\bar{T}_B\circ \alpha(x)$ for every point $x\in X_C$. By Proposition \ref{factor1}(3), this equation holds for every point $x\in \alpha^{-1}(X_B\setminus X_B^{\rm max})$. Now suppose that $x\in X_C^{\rm max}$. Then $\alpha(x)\in X_B^{\rm max}$ by Proposition \ref{factor1}(1). 
Let $z$ be the unique min path of $X_B$. Since $\bar{T}_B$ is a natural extension $\bar{T}_B(\alpha(x))=z$. Moreover, as $\bar{T}_C$ is also a natural extension, one can conclude that $\bar{T}_C(x)\in X_C^{\rm min}$. Hence, $\alpha\bar{T}_C(x)\in X_B^{\rm min}=\{z\}$ by Proposition \ref{factor1}(2). Thus $\alpha\circ\bar{T}_C(x)=\bar{T}_B\circ\alpha(x)$.

Now assume that $x\in (X_C\setminus X_C^{\rm max})\setminus\alpha^{-1}(X_B\setminus X_B^{\rm max})$. Then $\alpha(x)\in X_B^{\rm max}$ but $x\notin X_C^{\rm max}$. From the definition of $\alpha$ it follows that $\alpha T_C(x)$ is a min path. 
Therefore, $\alpha T_C(x)=z$ and as $\bar{T}_B$ is a natural extension we have $\bar{T}_B(\alpha(x))=z$. In conclusion, $\alpha T_C(x)=T_B\alpha(x)$. 
\end{proof}

\begin{definition}
Let $B$ be an ordered Bratteli diagram. An infinite path $w=e_1e_2\cdots$ in $X_B$ is said to be {\it  eventually maximal} (resp., {\it minimal}) if there exists an infinite max (resp., minimal) path in its forward (resp., backward) orbit. In other words,  there exists $m\geq 1$ such that $e_n={\rm maximal}$ (resp., {\rm mininmal}) edge for every $n\geq m$.
\end{definition}

\begin{lemma}\label{lemmain}
Let $B$ be an ordered Bratteli diagram. Suppose that $B'$, $f:B\rightarrow B'$, $z\in X_B^{\rm min}$, $y\in X_B^{\rm max}$,  and $\alpha=\mathcal V(f):X_{B'}\rightarrow X_B$ are  as in Subsection \ref{construction}. Let $\bar T_B:X_B\rightarrow X_B$ be a continuous extension of the Vershik map $T_B:X_B\setminus X_B^{\rm max}\rightarrow X_B\setminus X_B^{\rm min}$. Then 
\begin{enumerate} 
\item\label{existence} there is a continuous extension $\bar T_{B'}:X_{B'}\rightarrow X_{B'}$ of the Vershik map $T_{B'}:X_{B'}\setminus X_{B'}^{\rm max}\rightarrow X_{B'}\setminus X_{B'}^{\rm min}$ such that
$\bar T_{B'}\upharpoonright_{X_B}=\bar T_B$.
\item\label{unique} If $z$ is not eventually maximal  then $X_{B'}^{\rm max}=X_B^{\rm max}$ and $\bar T_{B'}$ in (1) is unique.
\item If $z$ is eventually maximal  then $X_{B'}^{\rm max}\setminus X_B^{\rm max}$ is a singleton, say $\{w_0\}$, and  for any other extension $S_{B'}$ of $T_{B'}$, $S_{B'}(w)=\bar T_{B'}(w)$ for all $w\in X_{B'}\setminus\{w_0\}$.
\item \label{semi} If $B$ is semi-decisive then so is $B'$.
\item $\alpha(\bar T_{B'}(w))=\bar T_B(\alpha(w))$ for all $w\in X_{B'}\setminus\{x\}$.

\end{enumerate}
\end{lemma}
\begin{proof}
We prove (1). Consider the Vershik homeomorphism $T_{B'}:X_{B'}\setminus X_{B'}^{\rm max}\rightarrow X_{B'}\setminus X_{B'}^{\rm min}$. Note that $B$ is a subdiagram of $B'$ and there are no edges in $E'$ with source in $V'\setminus V$ and range in $V$. Thus the successor of any edge or finite path in $E$ is the same as in $E'$ which turns out to have  $X_B^{\rm max}\subseteq X_{B'}^{\rm max}$ and $X_B^{\rm min} \subseteq X_{B'}^{\rm min}$. So $T_{B'}$ is an extension of $T_B:X_B\setminus X_B^{\rm max}\rightarrow X_B\setminus X_B^{\rm min}$. Therefore, for every point $x$  at which both $T_{B'}$ and $\bar T_B$ are defined, $T_{B'}(x)=\bar T_B(x)$. By the construction of $B'$ there are two possibilities for the point $z$:
\begin{enumerate}
\item[]Case I: If $z$ is not eventually maximal then by the arguments in Construction  we will have $X_{B'}^{\rm max}=X_B^{\rm max}$ and then
$X_{B'}=(X_{B'}\setminus X_{B'}^{\rm max})\cup X_B$.  So we define $\bar T_{B'}=T_B\cup \bar T_B$. 

\item[] Case II: If  $z$ is eventually maximal then $X_{B'}^{\rm max}=X_B^{\rm max}\cup \{w_0\}$.
Note that $w_0$ is an isolated point of $X_{B'}$ and 
$$X_{B'}=(X_{B'}\setminus X_{B'}^{\rm max})\cup X_B^{\rm max}\cup \{w_0\}.$$
We define $\bar T_{B'}=T_{B'}\cup \bar T_B$ on $X_{B'}\setminus \{w_0\}$ and $\bar T_{B'}(w_0)=\bar{T}_B(\alpha(w_0))$. (Note that any point chosen from $X_{B'}$ for defining $\bar{T}_{B'}(w_0)$ will give us a continuous extension of $\bar{T}_{B'}$.)
\end{enumerate}
Now we prove that $\bar T_{B'}$ is continuous. As the Vershik map $T_{B'}$ is a homeomorphism and $X_{B'}\setminus X_{B'}^{\rm max}$ is open in $X_{B'}$, it is enough to prove the continuity of $\bar T_{B'}$ at any $w\in X_{B'}^{\rm max}$. Note that in Case  II, since $w_0$ is an isolated point, $\bar T_{B'}$ is continuous at it. So  we are left to prove the continuity of $\bar T_{B'}$ at any point $w\in X_B^{\rm max}$. For this, it is enough to show that if $(w_n)_{n=1}^\infty$ is a sequence in $X_{B'}\setminus X_B$ converging to some $w\in X_B^{\rm max}$, then $T_{B'}(w_n)\rightarrow \bar T_B(w)$ (see Figure~\ref{fig3}). 
Without loss of generality we may assume that the first $n$ edges of $w_n$ are in $E'$ with source in $V_n$ and with range equal to $v_0^{n+1}$, and for every $k>n+1$ the $k$-th edge of $w_n$ is the $k$-th edge of $x$ (i.e. $x_k$). Let $w_n=w_n^1w_n^2\cdots$ where $w_n^j\in E'_j$ for all $j\in \mathbb N$. We consider $B$ as a subdiagram of $B'$ and so $\alpha(w_n)\in X_{B'}$ for all $n\in\mathbb N$. 

\begin{figure}
\begin{center}
\begin{tikzpicture}[scale=1]

\filldraw[] (8,10) circle [radius=0.1];
 \filldraw[] (7,8) circle [radius=0.1];
\filldraw[] (9,8) circle [radius=0.1];
\node at (9.9,8.7) {$w^1, w^1_1, w_2^1, w_3^1$};
\node at (9.9,6.8) {$w^2, w^2_2, w_3^2$};
\node at (9.6,4.8) {$w^3, w^3_3$};
 \filldraw[] (7,6) circle [radius=0.1];
\filldraw[] (9,6) circle [radius=0.1];
 \filldraw[] (7,4) circle [radius=0.1];
\filldraw[] (9,4) circle [radius=0.1];
 \filldraw[] (7,2) circle [radius=0.1];

\draw[] (7.93,9.98)--(6.98,8.12);

\draw[] (8.085,9.915)--(8.98,8.12);

\draw[] (8.98,8.0)--(7.0,6.0);

\draw[] (6.98,8.0)--(6.98,4.14);

\draw[] (6.98,4.0)--(6.98,2.14);

\draw[] (8.98,8.12)--(9.0,6.14);

\draw[] (8.98,6.0)--(8.98,4.14);

\draw[] (8.98,6.0)--(6.98,4.0);

\draw[] (9.0,4.0)--(7.08,2.1);

\node at (7.40,2.65) {\tiny{$w^4_3$}};
\node at (7.40,4.65) {\tiny{$w_2^3$}};
\node at (7.40,6.65) {\tiny{$w_1^2$}};

\end{tikzpicture}
\end{center}
\caption{Three initial edges $w^j$, $j=1, 2,3$ of the infinite path $w$ as well as some initial edges of three infinite paths $w_i$, $i=1,2,3$ are depicted. In this figure, $w^j_i$ shows the $j$-th edge of the infinite path $w_i$.}\label{fig3}
\end{figure}
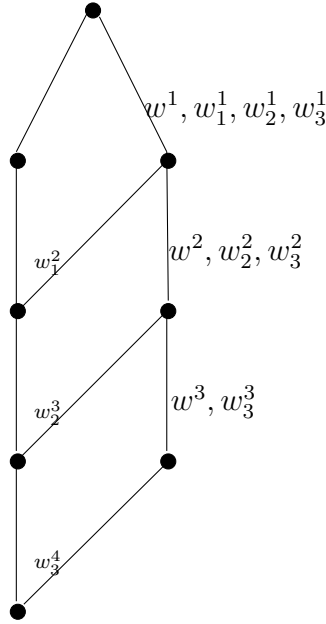

{\bf Claim.} $T_{B'}(w_n)$ and $T_B(\alpha(w_n))$ have the same first $n$ edges for all $n\in\mathbb N$.

To prove the claim let $n\in\mathbb N$. There are two cases here:
either $w_n^{n+1}<x_{n+1}$ or $w_n^{n+1}>x_{n+1}$. 

If $w_n^{n+1}<x_{n+1}$ then by order commutativity of the ordered premorphism $f$, $w_n^{n+1}=h'_{n+1,j}$ which corresponds to some $h_{n+1,j}$ with $j<m_n$ in Construction. Thus $s(h_{n+1,j})=s(w_n^{n+1})$ and 
$$\alpha(w_n)=w_n^1w_n^2\cdots w_n^nh_{n+1,j}y_{n+2}y_{n+3}\cdots.$$
Then $T_{B'}w_n$ and $T_B(\alpha(w_n))$ have the same first edges (see Figure~\ref{fig1}).

\begin{figure}
\begin{center}
\begin{tikzpicture}[scale=1.1]
{\color{blue}

\filldraw (10,4) circle [radius=0.1];
\filldraw (10,6) circle [radius=0.1];
\filldraw (10,2) circle [radius=0.1];
\filldraw (12,6) circle [radius=0.1];
\draw (10,6)--(10,2);
\draw (12,6)--(10,4);

\node at (11.45,5) {\tiny{$w_{n}^{n+1}$}};

\node at (9.7,5) {\tiny{$x_{n+1}$}};

\node at (10.1,3) {\tiny{$w_n^{n+2}=x_{n+2}$}};
}
{\color{red}
\filldraw (5,6) circle [radius=0.1];
\filldraw(6,6) circle [radius=0.1];
\filldraw(7,2) circle [radius=0.1];

\filldraw (7,6) circle [radius=0.1];
 \filldraw (5,4) circle [radius=0.1];
\filldraw (7,4) circle [radius=0.1];

\draw (4.95,6.0)--(4.95,4.14);
\node at (7.35,5) {\tiny{$y_{n+1}$}};
\draw (7.03,6.0)--(7.03,4.14);
\draw (6,6)--(7,4);
\draw (7,4)--(7,2);
\node at (7.35,3) {\tiny{$y_{n+2}$}};

\node at (6.1,5) {\tiny{$h_{n+1,j}$}};
\node at (4.6,5) {\tiny{$z_{n+1}$}};

}
\draw[->, thick] (5.15,4.05) [out=10, in=170] to (9.85,4.1);
\draw[->, thick] (7.15,3.95) [out=-10, in=190] to (9.85,3.88);

\node at (9.6,4.3) {\tiny{1}};
\node at (9.6,3.99) {\tiny{0}};

\end{tikzpicture}
\end{center}
\caption{The case $w_n^{n+1}<x_{n+1}.$}
\label{fig1}
\end{figure}
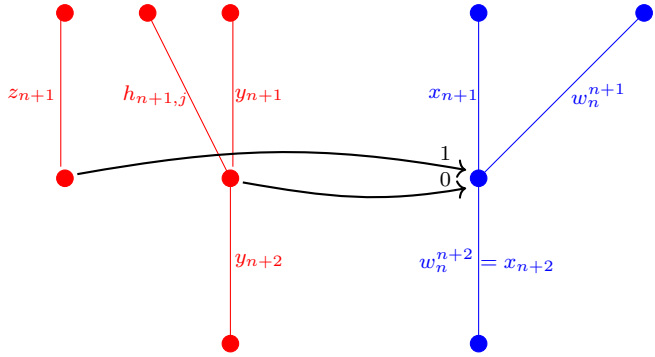
\medskip

 In the second case that $w_n^{n+1}>x_{n+1}$ we have $w_n^{n+1}=g'_{n+1,j}$ which corresponds  to some $g_{n+1,j}$ with $j\leq \ell_n$ in Construction (see Figure~\ref{fig2}). This means that $s(g_{n+1,j})=s(w_n^{n+1})$  (see Figure~\ref{fig2}) and 
 $$\alpha(w_n)=w_n^1w_n^2\cdots w_n^ng_{n+1,j}z_{n+2}z_{n+3}\cdots.$$
 
 \begin{figure}
\begin{center}
\begin{tikzpicture}[scale=1.1]
{\color{blue}

\filldraw (10,4) circle [radius=0.1];
\filldraw (10,6) circle [radius=0.1];
\filldraw (10,2) circle [radius=0.1];
\filldraw (12,6) circle [radius=0.1];
\draw (10,6)--(10,2);
\draw (12,6)--(10,4);

\node at (11.45,5) {\tiny{$w_{n}^{n+1}$}};

\node at (9.7,5) {\tiny{$x_{n+1}$}};

\node at (10.1,3) {\tiny{$w_n^{n+2}=x_{n+2}$}};
}
{\color{red}
\filldraw (5,6) circle [radius=0.1];
\filldraw(6,6) circle [radius=0.1];
\filldraw(7,2) circle [radius=0.1];

\filldraw (7,6) circle [radius=0.1];
 \filldraw (5,4) circle [radius=0.1];
\filldraw (7,4) circle [radius=0.1];

\draw (4.95,6.0)--(4.95,4.14);
\node at (7.35,5) {\tiny{$y_{n+1}$}};
\draw (7.03,6.0)--(7.03,4.14);
\draw (6,6)--(5,4);
\draw (7,4)--(7,2);
\node at (7.35,3) {\tiny{$y_{n+2}$}};

\node at (6.1,5) {\tiny{$g_{n+1,j}$}};
\node at (4.6,5) {\tiny{$z_{n+1}$}};

}
\draw[->, thick] (5.15,4.05) [out=10, in=170] to (9.85,4.1);
\draw[->, thick] (7.15,3.95) [out=-10, in=190] to (9.85,3.88);

\node at (9.6,4.3) {\tiny{1}};
\node at (9.6,3.99) {\tiny{0}};

\end{tikzpicture}
\end{center}
\caption{The case $w_n^{n+1}>x_{n+1}.$}
\label{fig2}
\end{figure}
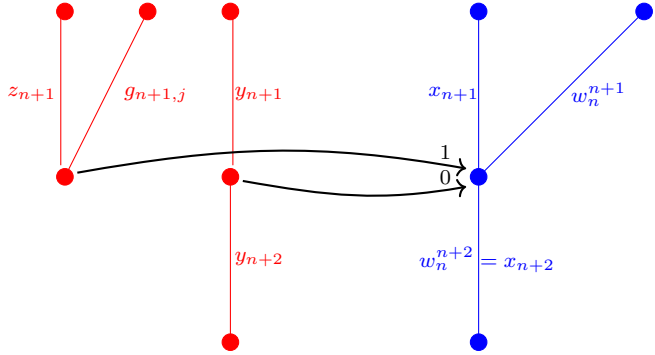

 When $j<\ell_n$, the $(n+1)$-th edges of $T_{B'}(w_n)$ and $T_B(\alpha(w_n))$ are $g'_{n+1,j+1}$ and $g_{n+1,j+1}$, respectively, and the first $n$ edges of $T_{B'}(w_n)$ and $T_B(\alpha(w_n))$ are the same. If $j=\ell_n$ then as $T_{B'}(x)$ is not an infinite max path, by Theorem \ref{base}(4) we see that $z\notin X_B^{\rm max}$ and so there is $t>n+1$ such that $x_t$ is not a max edge (equivalently, $\ell_t>1$) and $x_{n+2},\ldots, x_{t-1}$ are max edges (see Figure~\ref{fig5}).

\begin{figure}
\begin{center}
\begin{tikzpicture}[scale=1]

 \filldraw[] (7,8) circle [radius=0.1];
\filldraw[] (9,8) circle [radius=0.1];
 \filldraw[] (7,6) circle [radius=0.1];
 \filldraw[] (7,4) circle [radius=0.1];
\filldraw[] (9,4) circle [radius=0.1];
 \filldraw[] (7,2) circle [radius=0.1];


\draw[] (8.98,8.0)--(7.0,6.0);

\draw[] (6.98,8.0)--(6.98,4.14);

\draw[] (6.98,3.2)--(6.98,2.14);

\draw[] (9.0,4.0)--(7.08,2.1);

\node at (6.70,6.65) {\tiny{$x_{n+1}$}};
\node at (6.70,4.65) {\tiny{$x_{n+2}$}};
\node at (6.70,2.65) {\tiny{$x_{t}$}};

\node at (6.98,3.7) {\tiny{$\vdots$}};

\node at (8.40,6.65) {\tiny{$w_{n}^{n+1}={\rm max}$}};

\end{tikzpicture}
\end{center}
\caption{}\label{fig5}
\end{figure}

  Then the $t$-th edges of $T_{B'}(w_n)$ and $T_B(\alpha(w_n))$ are $g'_{t,2}$ and $g_{t,2}$ respectively, and $T_{B'}(w_n)$ and $T_B(\alpha(w_{n+1}))$ have the same first $(t-1)$ edges and hence first $n$ edges. This finishes the proof of the claim.
 
 \smallskip
 
 Observe that, by the definition of $\alpha$, the first $n$ edges of $w_n$ and $\alpha(w_n)$ are the same. Therefore, $\alpha(w_n)\rightarrow w$ in $X_B$ since $w_n\rightarrow w$ in $X_{B'}$. Consequently, $T_B(\alpha(w_n))\rightarrow \bar T_B(w)$. On the other hand, the claim implies that 
 $d(T_{B'}(w_n),T_B(\alpha(w_n)))\rightarrow 0$ where $d$ is the canonical metric on $X_{B'}$. Hence, $T_{B'}(w_n)\rightarrow \bar T_B(w)$. This finishes the proof of part (1).
 
 Parts (2) and (3) were covered in  Case I and Case II.
 Part (4) follows from parts (1)-(3). 
 To prove part (5), let $w\in X_{B'}\setminus \{x\}$. Consider the following three cases:
 \begin{itemize}
 \item [(a)] $w\in X_B$. In this case $\alpha(w)=w$ and $$\bar{T}_{B}(\alpha(w))=\bar{T}_{B}(w)=\bar{T}_{B'}(w)=\alpha(\bar{T}_{B'}(w)).$$
 \item [(b)] $w\in X_{B'}\setminus X_B$ and $w\notin X_{B'}^{\rm max}$. Then $\alpha(w)\notin X_{B}^{\rm max}$ (since $w\neq x$). By Proposition \ref{factor1}\eqref{intersect} the desired equation is satisfied. 
 \item [(c)] $w\in X_{B'}\setminus X_B$ and $w\in X_{B'}^{\rm max}$. This case happens only in Case II which means that $w=w_0$ and by definition of $\bar{T}_{B'}(w_0)$ and the fact that $\alpha\upharpoonright_{X_B}={\rm id}$, we have
$$\bar{T}_{B}(\alpha(w_0))=\bar{T}_{B'}(w_0)=\alpha(\bar{T}_{B'}(w_0)).$$
 \end{itemize}
  This finishes the proof.
\end{proof}

\begin{proposition}\label{dc}
Let $B$ be a decisive  ordered Bratteli diagram.  Suppose that $B'$, $f:B\rightarrow B'$, $z\in X_B^{\rm min}$, $y\in X_B^{\rm max},$  and $\alpha=\mathcal V(f):X_{B'}\rightarrow X_B$ are  as in Subsection \ref{construction}. Then $B'$ is decisive if and only if one of the following statements holds:
\begin{enumerate}
\item $z$ is not eventually maximal and $y$ is not eventually minimal. (See Figure~\ref{fig7} as an example.)
\item $z$ is eventually maximal,  $y$ is eventually minimal, and $X_B^{\rm max}$ has empty interior. (See Figure~\ref{fig8} as an example.)
\end{enumerate}
In particular, If $B$ is simple and decisive such that $X_B$ is infinite then $B'$ is decisive.
\end{proposition}
\begin{proof}
First observe that $z$ (resp., $y$) is eventually maximal (resp., minimal) if and only if $T_{B'}(x)$ (resp., $x$) is eventually maximal (resp., minimal). So for the backward implication first assume that (1) holds. Then $x$ and $T_{B'}(x)$ are not eventually maximal and eventually minimal, respectively, which means that $X_{B'}^{\rm max}=X_B^{\rm max}$ and  $X_{B'}^{\rm min}=X_B^{\rm min}$. 
Let $\bar{T}_B:X_B\rightarrow X_B$ be the unique homeomorphism extension of $T_B:X_B\setminus X_B^{\rm max}\rightarrow X_B\setminus X_B^{\rm min}$. By Lemma \ref{lemmain}\eqref{existence} there is a continuous map $\bar{T}_{B'}:X_{B'}\rightarrow X_{B'}$ extending both $\bar{T}_B$ and the Vershik map $T_{B'}:X_{B'}\setminus X_{B'}^{\rm max}\rightarrow X_{B'}\setminus X_{B'}^{\rm min}$. Since $$X_{B'}\setminus X_B\subseteq X_{B'}\setminus (X_{B'}^{\rm max}\cup X_{B'}^{\rm min}),$$
it follows that $\bar{T}_{B'}(X_B)=X_B$ and $\bar{T}_{B'}(X_{B'}\setminus X_B)=X_{B'}\setminus X_B$. Hence, $\bar{T}_{B'}$ is a homeomorphism extension of the Vershik map $T_{B'}$. Uniqueness of $\bar{T}_{B'}$ is clear by Lemma \ref{lemmain}\eqref{unique}. Therefore, $B'$ is decisive. 
\medskip

Now assume that (2) holds. Then $x$ and $T_{B'}(x)$ are eventually minimal and eventually maximal, respectively. As $x$ and $T_{B'}(x)$ are cofinal, $x$ is eventually maximal too. It turns out that
$$\exists m\geq 1\ \forall n\geq m\  \ r^{-1}(r(x_n))=\{x_n\}.$$
and the latter means that $X_{B'}\setminus X_B$ is a cycle, i.e., there exist $k\geq 2$ and some  $w_1,\ldots, w_k\in X_{B'}$ so that $w_1\in X_{B'}^{\rm min}$, $w_k\in X_{B'}^{\rm max}$, $w_{i+1}$ is the successor of $w_i$ for $i=1,\ldots, k-1$ and $$X_{B'}\setminus X_B=\{w_1, \ldots, w_k\}.$$
Note that there is $1\leq j<k$ such that $w_j=x$ and $w_{j+1}=T_{B'}(x)$. Moreover, each of the $w_i$'s are isolated points in $X_{B'}$. Therefore, one can define a homeomorphism extension, say $\bar{T}_{B'}$, of 
$T_{B'}:X_{B'}\setminus X_{B'}^{\rm max}\rightarrow X_{B'}\setminus X_{B'}^{\rm min}$ by Lemma \ref{lemmain} and letting $\bar{T}_{B'}(w_k)=w_1$. We claim that this extension is unique. Since otherwise, if $S:X_{B'}\rightarrow X_{B'}$ is another extension of $T_{B'}$, then there exists some $t\in X_B^{\rm min}$ such that $T_{B'}(w_k)=t$ and $T_{B'}(T^{-1}_B(t))=w_1$. By the assumption, $T^{-1}_B(t)\in X_B^{\rm max}$ is not an isolated point in $X_B$ (and hence in $X_{B'}$ by Theorem \ref{base}\eqref{isolated}) while $x$ is an isolated point of  $X_{B'}$, a contradiction. Therefore, $B'$ is decisive. 
\medskip

For the forward implication, assume that $B'$ is decisive and that (1) does not hold. We prove that (2) is satisfied. Let $z$ be eventually maximal. Suppose that $y$ is not eventually minimal. Then $T_{B'}(x)$ is eventually maximal and $x$ is not eventually minimal. So $X_{B'}\setminus X_B$ is a singleton, say $\{w\}$, while $X_{B'}^{\rm min}=X_B^{\rm min}$. Since $B'$ is decisive, $\bar{T}_{B'}(w)\in X_B^{\rm min}$ and it is isolated in $X_B$ (as $w$ is isolated in $X_{B'}$). Since $B$ is decisive, $\bar{T}_B^{-1}(\bar{T}_B(w))\in X_B^{\rm max}$ and is an isolated point of $X_B$. By Theorem \ref{base}\eqref{isolated} this point will be isolated in $X_{B'}$. Hence, $y$ is eventually minimal. Similarly, if $y$ is eventually minimal then $z$ is eventually maximal. Since we had the assumption of not having (1), the conclusion is that $z$ is eventually maximal and $y$ is eventually minimal. So it remains to show that $X_B^{\rm max}$ has empty interior. Suppose that $$X_{B'}=X_B\cup\{w_1,\ldots, w_k\}$$ as above. Recall that each $w_i$ is isolated. If $X_B^{\rm max}$ has non-empty interior, then (by decisiveness) so does $X_B^{\rm min}$. Let $t\in X_B^{\rm min}$ be an isolated point of $X_B$ (and hence $X_{B'}$). Then one can define $S:X_{B'}\rightarrow X_{B'}$ such that $S(w_k)=t$, $S(T_B^{-1}(t))=w_1$ and $S(w)=T_B(w)$   for all other $w\in X_B$ which is clearly another extension of $\bar{T}_{B'}$, a contradiction. Consequently, (2) holds.

\medskip

If $B$ is simple then (1) happens and so $B'$ is decisive. 
\end{proof}

\begin{corollary}
Let $B$ be a properly ordered Bratteli diagram, $z=x_{\rm min}$ and $y=x_{\rm max}$. Let $B'$, $x\in X_{B'}$ and $f:B\rightarrow B'$ be as in Subsection \ref{construction}. Then $X_{B'}$ is decisive if and only if $\#X_B=\infty$.
\end{corollary}
\begin{proof}
If $B'$ is decisive, then each of the two  statements of Proposition \ref{dc} imply that $X_B$ is infinite. 

Now suppose that $X_B$ is infinite. We prove that if the statement (1) of Proposition \ref{dc} does not hold, then (2) is satisfied.  To show this, without loss of generality, assume that $z=z_1z_2\cdots$ is eventually maximal. Choose $m\geq 1$ such that $z_n$ is a maximal edge for all $n\geq m$. Let $z'_1z'_2\ldots z'_{m-1}$ be the finite maximal path (from $v_0$ to $s(z_m)$) such  that $w=z'_1z'_2\ldots z'_{m-1}z_mz_{m+1}\ldots\in X_B^{\rm max}$. By properness, $w=y$ and therefore, $y$ is eventually minimal. It remains to show that $y$ is not isolated. But if $y$ is isolated then there exists some $k\geq m$ such that 
 for all $n>k$, $r^{-1}(r(z_n))=s^{-1}(s(z_n))$. As $X_B$ is infinite, there are at least two vertices at each level $n>k$ which means that there are (infinite) min paths and (infinite) max paths other than $z$ and $y$, contradicting properness.
\end{proof}
\medskip

Now we have all  tools in hand to prove Theorem \ref{main} about relation between ordered premorphisms and topological factoring for decisive ordered Bratteli diagrams. 

\begin{proof}[\bf Proof of Theorem \ref{main}.]
$(2)\Rightarrow (1)$ is true by Proposition \ref{factor3}. For the converse suppose that $B$ has at least two infinite min paths. Let $y$ be an infinite max path in $B$ and $z\in X_B^{\rm min}\setminus \{\bar{T}_B(y)\}$ to make a Bratteli diagram $B'$ with an ordered premorphism $f$ between $B$ and $B'$ as in  Construction. By Lemma \ref{lemmain}\eqref{semi}, $C:=B'$ is a semi-decisive ordered Bratteli diagram such that by Theorem \ref{base}\eqref{nofactor}, $\mathcal{V}(f)$ is not a topological factoring from $X_C$ to $X_B$.
\end{proof}

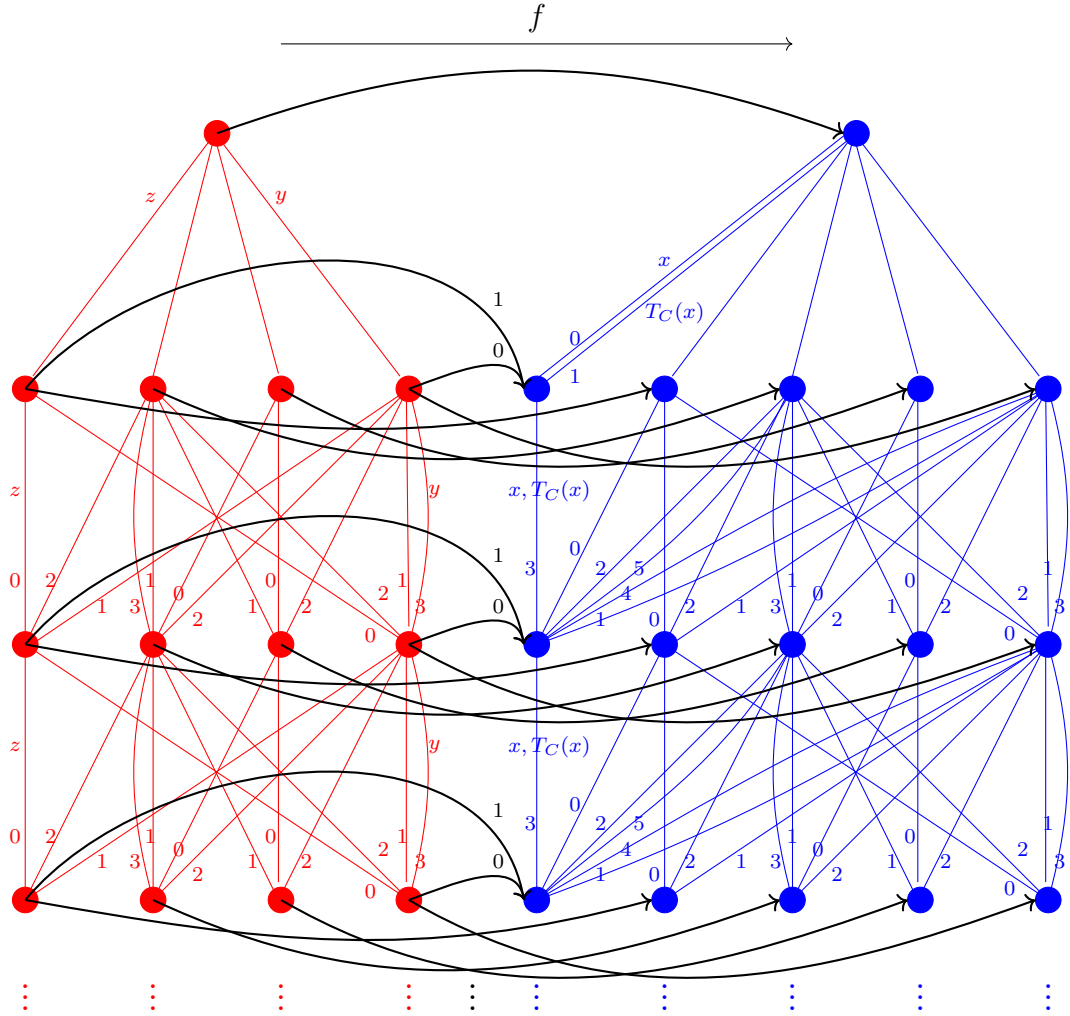
\begin{figure}
\begin{center}
\begin{tikzpicture}[scale=1.7]
{\color{red}
\filldraw[] (2.5,10) circle [radius=0.1];
 \filldraw[] (3,8) circle [radius=0.1];
\filldraw[] (4,8) circle [radius=0.1];
 \filldraw[] (3,6) circle [radius=0.1];
\filldraw[] (4,6) circle [radius=0.1];
 \filldraw[] (3,4) circle [radius=0.1];
\filldraw[] (4,4) circle [radius=0.1];
 \filldraw[] (2,8) circle [radius=0.1];
\filldraw[] (2,6) circle [radius=0.1];
\filldraw[] (2,4) circle [radius=0.1];

\filldraw[] (1,8) circle [radius=0.1];
 \filldraw[] (1,6) circle [radius=0.1];
\filldraw[] (1,4) circle [radius=0.1];


\draw[] (4,6) [out=75, in=-75] to (4,8);

\draw[] (4,4) [out=75, in=-75] to (4,6);

\draw[] (2,6) [out=105, in=-105] to (2,8);

\draw[] (2,4) [out=105, in=-105] to (2,6);

\draw[] (2.48,9.98)--(2.98,8.12);

\draw[] (2.5,9.98)--(3.93,8.10);

\draw[] (3.98,7.98)--(3.0,6.0);

\draw[] (2.98,7.95)--(2.98,4.14);


\draw[] (3.98,8.0)--(4.0,6.14);

\draw[] (3.98,6.0)--(3.98,4.14);

\draw[] (3.98,6.0)--(2.98,4.0);

\draw[] (1.98,8.0)--(1,6);

\draw[] (2.48,9.98)--(2,8.1);

\draw[] (2,8)--(2,4);

\draw[] (2,8)--(2.98,6);

\draw[] (2,6)--(2.98,4);

\draw[] (4,8)--(1.05,6);

\draw[] (2.47,9.98)--(1.06,8.1);
\draw[] (1,8)--(1,4);
\draw[] (3.98,6)--(1.0,4);

\draw[] (1,8)--(4,6);
\draw[] (1,6)--(4,4);
\draw[] (2,6)--(4,4);

\draw[] (2,8)--(4,6);

\draw[] (3,8)--(2,6);
\draw[] (4,8)--(2,6);
\draw[] (4,6)--(2,4);

\draw[] (3,6)--(2,4);

\draw[] (2,6)--(1,4);

%
%
\node at (0.92,6.5) {\tiny{$0$}};
\node at (0.92,4.5) {\tiny{$0$}};
\node at (2.92,6.5) {\tiny{$0$}};
\node at (2.92,4.5) {\tiny{$0$}};
\node at (2.2,6.4) {\tiny{$0$}};
\node at (2.2,4.4) {\tiny{$0$}};
\node at (3.7,4.07) {\tiny{$0$}};
\node at (3.7,6.07) {\tiny{$0$}};
\node at (2.78,6.3) {\tiny{$1$}};
\node at (2.78,4.3) {\tiny{$1$}};
\node at (1.6,6.3) {\tiny{$1$}};
\node at (1.6,4.3) {\tiny{$1$}};

\node at (2.35,6.2) {\tiny{$2$}};

\node at (2.35,4.2) {\tiny{$2$}};

\node at (3.8,4.4) {\tiny{$2$}};

\node at (3.8,6.4) {\tiny{$2$}};

\node at (1.2,6.5) {\tiny{$2$}};
\node at (1.2,4.5) {\tiny{$2$}};
\node at (3.2,6.3) {\tiny{$2$}};

\node at (1.98,6.5) {\tiny{$1$}};
\node at (1.86,6.3) {\tiny{$3$}};

\node at (1.98,4.5) {\tiny{$1$}};
\node at (1.86,4.3) {\tiny{$3$}};

\node at (3.2,4.3) {\tiny{$2$}};

\node at (4.09,6.3) {\tiny{$3$}};
\node at (3.95,6.5) {\tiny{$1$}};

\node at (4.09,4.3) {\tiny{$3$}};
\node at (3.95,4.5) {\tiny{$1$}};

}
{\color{blue}

\filldraw[] (7.5,10) circle [radius=0.1];
 \filldraw[] (8,8) circle [radius=0.1];
\filldraw[] (9,8) circle [radius=0.1];
 \filldraw[] (8,6) circle [radius=0.1];
\filldraw[] (9,6) circle [radius=0.1];
 \filldraw[] (8,4) circle [radius=0.1];
\filldraw[] (9,4) circle [radius=0.1];
\filldraw[] (7,8.0) circle [radius=0.1];
\filldraw[] (7,6) circle [radius=0.1];
\filldraw[] (7,4) circle [radius=0.1];

\filldraw[] (6,8) circle [radius=0.1];
 \filldraw[] (6,6) circle [radius=0.1];
\filldraw[] (6,4) circle [radius=0.1];

\filldraw[] (5,8) circle [radius=0.1];
\filldraw[] (5,6) circle [radius=0.1];
\filldraw[] (5,4) circle [radius=0.1];

 
 \draw[] (9,6) [out=75, in=-75] to (9,8);

\draw[] (9,4) [out=75, in=-75] to (9,6);

\draw[] (7,6) [out=105, in=-105] to (7,8);

\draw[] (7,4) [out=105, in=-105] to (7,6);

\draw[] (5,6) [out=30, in=-160] to (9,8);
\draw[] (5,6) [out=20, in=-150] to (9,8);

\draw[] (5,4) [out=30, in=-160] to (9,6);
\draw[] (5,4) [out=20, in=-150] to (9,6);
 
\draw[] (7.48,9.98)--(7.98,8.12);

\draw[] (7.5,9.98)--(8.93,8.10);

\draw[] (8.98,7.98)--(8.0,6.0);

\draw[] (7.98,7.95)--(7.98,4.14);

\draw[] (8.98,8.0)--(9.0,6.14);

\draw[] (8.98,6.0)--(8.98,4.14);

\draw[] (8.98,6.0)--(7.98,4.0);

\draw[] (6.98,8.0)--(6,6);

\draw[] (7.48,9.98)--(7,8.1);

\draw[] (7,8)--(7,4);

\draw[] (7,8)--(7.98,6);

\draw[] (7,6)--(7.98,4);

\draw[] (9,8)--(6.05,6);

\draw[] (7.47,9.98)--(6.06,8.1);
\draw[] (6,8)--(6,4);
\draw[] (8.98,6)--(6.0,4);

\draw[] (6,8)--(9,6);
\draw[] (6,6)--(9,4);
\draw[] (7,6)--(9,4);

\draw[] (7,8)--(9,6);

\draw[] (8,8)--(7,6);
\draw[] (9,8)--(7,6);
\draw[] (9,6)--(7,4);

\draw[] (8,6)--(7,4);

\draw[] (7,6)--(6,4);

\draw[] (7.5,10)--(5,8);
\draw[] (7.4,9.98)--(4.9,8);

\draw[] (5,8)--(5,4);

\draw[] (7,8)--(5,6);
\draw[] (7,6)--(5,4);

\draw[] (6,8)--(5,6);
\draw[] (6,6)--(5,4);

\draw[] (7,8)[out=240, in=35] to (5,6);
\draw[] (7,6)[out=240, in=40] to (5,4);

%

\node at (5.92,6.2) {\tiny{$0$}};
\node at (5.92,4.2) {\tiny{$0$}};
\node at (7.92,6.5) {\tiny{$0$}};
\node at (7.92,4.5) {\tiny{$0$}};
\node at (7.2,6.4) {\tiny{$0$}};
\node at (7.2,4.4) {\tiny{$0$}};
\node at (8.7,4.09) {\tiny{$0$}};
\node at (8.7,6.09) {\tiny{$0$}};
\node at (7.78,6.3) {\tiny{$1$}};
\node at (7.78,4.3) {\tiny{$1$}};

\node at (6.6,6.3) {\tiny{$1$}};

\node at (6.6,4.3) {\tiny{$1$}};

\node at (7.35,6.2) {\tiny{$2$}};

\node at (7.35,4.2) {\tiny{$2$}};
\node at (8.8,4.4) {\tiny{$2$}};
\node at (8.8,6.4) {\tiny{$2$}};
\node at (6.2,6.3) {\tiny{$2$}};
\node at (6.2,4.3) {\tiny{$2$}};
\node at (8.2,6.3) {\tiny{$2$}};

\node at (7,6.5) {\tiny{$1$}};
\node at (6.87,6.3) {\tiny{$3$}};

\node at (7,4.5) {\tiny{$1$}};
\node at (6.87,4.3) {\tiny{$3$}};

\node at (8.2,4.3) {\tiny{$2$}};
\node at (9.09,4.3) {\tiny{$3$}};
\node at (9.09,6.3) {\tiny{$3$}};

\node at (9,4.6) {\tiny{$1$}};
\node at (9,6.6) {\tiny{$1$}};

\node at (5.3,6.75) {\tiny{$0$}};
\node at (5.3,4.75) {\tiny{$0$}};

\node at (5.5,6.6) {\tiny{$2$}};
\node at (5.5,4.6) {\tiny{$2$}};

\node at (4.95,6.6) {\tiny{$3$}};
\node at (4.95,4.6) {\tiny{$3$}};

\node at (5.5,6.199) {\tiny{$1$}};

\node at (5.5,4.199) {\tiny{$1$}};

\node at (5.8,6.6) {\tiny{$5$}};
\node at (5.8,4.6) {\tiny{$5$}};

\node at (5.7,6.4) {\tiny{$4$}};
\node at (5.7,4.4) {\tiny{$4$}};

\node at (6,9) {\tiny{$x$}};
\node at (5.3,8.4) {\tiny{$0$}};
\node at (5.3,8.1) {\tiny{$1$}};

\node at (6.08,8.6) {\tiny{$T_C(x)$}};
\node at (5.1,7.2) {\tiny{$x,T_C(x)$}};
\node at (5.1,5.2) {\tiny{$x, T_C(x)$}};

}

\draw[->, thick] (2.5,10) [out=20, in=160] to (7.4,10);
\draw[->, thick] (1,8) [out=350, in=195] to (5.9,8);
\draw[->, thick] (1,6) [out=350, in=195] to (5.9,6);
\draw[->, thick] (1,4) [out=350, in=195] to (5.9,4);
\draw[->, thick] (2,8) [out=335, in=200] to (6.9,8);
\draw[->, thick] (2,6) [out=335, in=200] to (6.9,6);
\draw[->, thick] (2,4) [out=335, in=200] to (6.9,4);
\draw[->, thick] (3,8) [out=330, in=200] to (7.9,8);
\draw[->, thick] (3,6) [out=330, in=200] to (7.9,6);
\draw[->, thick] (3,4) [out=330, in=200] to (7.9,4);
\draw[->, thick] (4,8) [out=330, in=200] to (8.9,8);
\draw[->, thick] (4,6) [out=330, in=200] to (8.9,6);
\draw[->, thick] (4,4) [out=330, in=200] to (8.9,4);
\draw[->, thick] (1,8) [out=50, in=100] to (4.9,8);
\draw[->, thick] (4,8) [out=20, in=100] to (4.9,8);
\draw[->, thick] (1,6) [out=50, in=100] to (4.9,6);
\draw[->, thick] (4,6) [out=20, in=100] to (4.9,6);
\draw[->, thick] (1,4) [out=50, in=100] to (4.9,4);
\draw[->, thick] (4,4) [out=20, in=100] to (4.9,4);

\node at (4.7,8.3) {\tiny{$0$}};
\node at (4.7,6.3) {\tiny{$0$}};
\node at (4.7,4.3) {\tiny{$0$}};

\node at (4.7,8.7) {\tiny{$1$}};
\node at (4.7,6.7) {\tiny{$1$}};
\node at (4.7,4.7) {\tiny{$1$}};

{\color{red}
\node at (0.92,7.2) {\tiny{$z$}};
\node at (0.92,5.2) {\tiny{$z$}};
\node at (1.98,9.5) {\tiny{$z$}};
\node at (4.2,7.2) {\tiny{$y$}};
\node at (4.2,5.2) {\tiny{$y$}};
\node at (3,9.5) {\tiny{$y$}};

}
{\color{red}
\node at (1,3.3) {\large{$\vdots$}};
\node at (2,3.3) {\large{$\vdots$}};
\node at (3,3.3) {\large{$\vdots$}};
\node at (4,3.3) {\large{$\vdots$}};
}
{\color{blue}
\node at (5,3.3) {\large{$\vdots$}};
\node at (6,3.3) {\large{$\vdots$}};
\node at (7,3.3) {\large{$\vdots$}};
\node at (8,3.3) {\large{$\vdots$}};
\node at (9,3.3) {\large{$\vdots$}};}
\node at (4.5,3.3) {\large{$\vdots$}};

\node at (5,10.9) {$f$};
\draw[->] (3,10.7) to (7,10.7);
\end{tikzpicture}
\end{center}
\caption{ The left diagram satisfies condition (1) of Proposition \ref{dc} and $\alpha(x)=y, \alpha(T_C(x))=z\neq T_B(y).$}\label{fig7}
\end{figure}
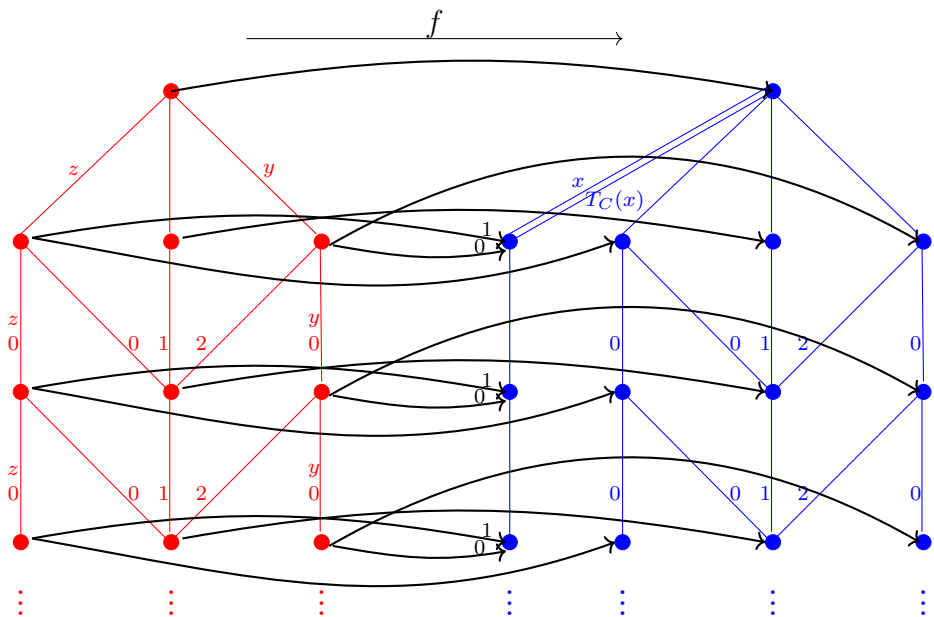
\begin{figure}
\begin{center}
\begin{tikzpicture}[scale=1]
{\color{red}
\filldraw[] (5,10) circle [radius=0.1];
 \filldraw[] (5,8) circle [radius=0.1];
\filldraw[] (7,8) circle [radius=0.1];
 \filldraw[] (5,6) circle [radius=0.1];
\filldraw[] (7,6) circle [radius=0.1];
 \filldraw[] (5,4) circle [radius=0.1];
\filldraw[] (7,4) circle [radius=0.1];
\filldraw[] (3,8.0) circle [radius=0.1];
\filldraw[] (3,6) circle [radius=0.1];
\filldraw[] (3,4) circle [radius=0.1];

\draw[] (4.98,9.98)--(4.98,8.12);

\draw[] (5.0,9.98)--(6.93,8.10);

\draw[] (6.98,7.98)--(5.0,6.0);

\draw[] (4.98,7.95)--(4.98,4.14);

\draw[] (6.98,8.0)--(7.0,6.14);

\draw[] (6.98,6.0)--(6.98,4.14);

\draw[] (6.98,6.0)--(4.98,4.0);

\draw[] (4.98,9.98)--(3,8.1);

\draw[] (3,8)--(3,4);

\draw[] (3,8)--(4.98,6);

\draw[] (3,6)--(4.98,4);


\node at (5.40,4.65) {\tiny{$2$}};
\node at (5.40,6.65) {\tiny{$2$}};

\node at (4.9,4.65) {\tiny{$1$}};
\node at (4.9,6.65) {\tiny{$1$}};

\node at (4.5,4.65) {\tiny{$0$}};
\node at (4.5,6.65) {\tiny{$0$}};

\node at (2.9,4.65) {\tiny{$0$}};
\node at (2.9,4.95) {\tiny{$z$}};

\node at (2.9,6.65) {\tiny{$0$}};
\node at (2.9,6.95) {\tiny{$z$}};
\node at (3.7,8.95) {\tiny{$z$}};
\node at (6.3,8.95) {\tiny{$y$}};

\node at (6.9,4.65) {\tiny{$0$}};
\node at (6.9,6.95) {\tiny{$y$}};
\node at (6.9,6.65) {\tiny{$0$}};
\node at (6.9,4.95) {\tiny{$y$}};

\node at (3,3.3) {\large{$\vdots$}};
\node at (7,3.3) {\large{$\vdots$}};
\node at (5,3.3) {\large{$\vdots$}};
}

{\color{blue}
\filldraw[] (9.5,8) circle [radius=0.1];
 \filldraw[] (9.5,6) circle [radius=0.1];
\filldraw[] (9.5,4) circle [radius=0.1];

\filldraw[] (13,10) circle [radius=0.1];
 \filldraw[] (13,8) circle [radius=0.1];
\filldraw[] (15,8) circle [radius=0.1];
 \filldraw[] (13,6) circle [radius=0.1];
\filldraw[] (15,6) circle [radius=0.1];
 \filldraw[] (13,4) circle [radius=0.1];
\filldraw[] (15,4) circle [radius=0.1];
\filldraw[] (11,8.0) circle [radius=0.1];
\filldraw[] (11,6) circle [radius=0.1];
\filldraw[] (11,4) circle [radius=0.1];


\draw[] (9.5,8)--(9.5,4);
\draw[] (12.98,9.99)--(9.5,8);
\draw[] (12.97,10.08)--(9.49,8.1);

\draw[] (12.98,9.98)--(12.98,8.12);

\draw[] (13.0,9.98)--(14.93,8.10);

\draw[] (14.98,7.98)--(13.0,6.0);

\draw[] (12.98,7.95)--(12.98,4.14);

\draw[] (14.98,8.0)--(15.0,6.14);

\draw[] (14.98,6.0)--(14.98,4.14);

\draw[] (14.98,6.0)--(12.98,4.0);

\draw[] (12.98,9.98)--(11,8.1);

\draw[] (11,8)--(11,4);

\draw[] (11,8)--(12.98,6);

\draw[] (11,6)--(12.98,4);


\node at (13.40,4.65) {\tiny{$2$}};
\node at (13.40,6.65) {\tiny{$2$}};

\node at (12.9,4.65) {\tiny{$1$}};
\node at (12.9,6.65) {\tiny{$1$}};

\node at (12.5,4.65) {\tiny{$0$}};
\node at (12.5,6.65) {\tiny{$0$}};

\node at (10.9,4.65) {\tiny{$0$}};
\node at (10.9,6.65) {\tiny{$0$}};

\node at (14.9,4.65) {\tiny{$0$}};
\node at (14.9,6.65) {\tiny{$0$}};

\node at (11,3.3) {\large{$\vdots$}};
\node at (15,3.3) {\large{$\vdots$}};
\node at (9.5,3.3) {\large{$\vdots$}};
\node at (13,3.3) {\large{$\vdots$}};
}
\draw[->, thick] (5,10) [out=10, in=170] to (13,10);
\draw[->, thick] (5.15,4.05) [out=10, in=170] to (12.9,4);
\draw[->, thick] (7.15,3.95) [out=-10, in=190] to (9.45,3.88);
\draw[->, thick] (5.15,6.05) [out=10, in=170] to (12.9,6);
\draw[->, thick] (5.15,8.05) [out=10, in=170] to (12.9,8);
\draw[->, thick] (3.15,8.05) [out=10, in=170] to (9.45,8);
\draw[->, thick] (3.15,8.05) [out=350, in=200] to (10.9,8);
\draw[->, thick] (3.15,4.05) [out=350, in=200] to (10.9,4);
\draw[->, thick] (3.15,6.05) [out=350, in=200] to (10.9,6);
\draw[->, thick] (7.15,5.95) [out=-10, in=190] to (9.45,5.88);
\draw[->, thick] (7.15,7.95) [out=-10, in=190] to (9.45,7.88);
\draw[->, thick] (3.15,6.05) [out=10, in=170] to (9.45,6);
\draw[->, thick] (3.15,4.05) [out=10, in=170] to (9.45,4);
\draw[->, thick] (7.1,3.95) [out=30, in=150] to (14.95,4);
\draw[->, thick] (7.1,5.95) [out=30, in=150] to (14.95,6);
\draw[->, thick] (7.1,7.95) [out=30, in=150] to (14.95,8);

\node at (9.2,4.16) {\tiny{$1$}};
\node at (9.1,3.94) {\tiny{$0$}};

\node at (9.2,6.16) {\tiny{$1$}};
\node at (9.1,5.94) {\tiny{$0$}};

\node at (9.2,8.16) {\tiny{$1$}};
\node at (9.1,7.94) {\tiny{$0$}};

{\color{blue} 
\node at (10.41,8.78) {\tiny{$x$}};
\node at (10.9,8.55) {\tiny{$T_C(x)$}};
}

\node at (8.5,10.9) {$f$};
\draw[->] (6,10.7) to (11,10.7);

\end{tikzpicture}
\end{center}
\caption{ The left diagram satisfies condition (2) of Proposition \ref{dc} and $\alpha(x)=y, \alpha(T_C(x))=z\neq T_B(y)=y.$}\label{fig8}
\end{figure}

\begin{corollary}\label{proper}
Let $B$ be a decisive simple ordered Bratteli diagram  such that $(X_B, T_B)$ is non-trivial. The following statements  are equivalent:
\begin{enumerate}
\item  for every decisive orederd Bratteli diagram $C$ and every ordered premorphism $f:B\rightarrow C$, the induced map $\mathcal V(f): X_C\rightarrow X_B$ is a topological factoring.
\item $B$ is proper.
\end{enumerate}
\end{corollary}
\begin{proof}
$(2)\Rightarrow (1)$ is a  direct consequence of Theorem \ref{main}. Conversely, Suppose that $B$ is not proper. Then there exist at least two infinite max paths, say $p, q$, and two infinite min paths, say $p'=\bar{T}_B(p), q'=\bar{T}_B(q)$ on it. 
Let $y:=p'$ and $z:=q$ to make diagram $B'$ and ordered premorphism $f: B\rightarrow B'$ as in  Construction. Simplicity of $B$ and then Proposition~\ref{dc} imply that   $B'$ is decisive. But by the choices of $z$ and $y$, $\mathcal V(f)$ is not topological factoring from $B'$ onto $B$ which contradicts $(1)$.
\end{proof}

\begin{remark}
Theorem \ref{main} shows that if $B$ and $C$ are decisive Bratteli diagrams and $f:B\rightarrow C$ is an ordered premorphism, then the induced map $\mathcal V(f): X_C\rightarrow X_B$ is not necessarily a topological factoring. In fact, for every decisive non-proper ordered Bratteli diagram $B$, using the method described in Subsection \ref{construction} and Proposition \ref{dc}, one can construct a semi-decisive Bratteli diagram  $C$ with ordered premorphism $f:B\rightarrow C$ without having factoring between the two Vershik systems. 
\end{remark}
\section{from topological factoring to ordered premorphisms}
 We are now going to   model  topological factorings $\alpha:(X,T)\rightarrow (Y,S)$ between  zero dimensional systems, by   sequences of ordered premorphisms. Let us recall from \cite{aeg21} that when the two systems are minimal and so both have realizations by properly ordered Bratteli diagrams (with unique maximal paths),  we fix two points $x_0\in X$ and $y_0\in Y$ that $\alpha(x_0)=y_0$ and then we construct ordered Bratteli diagrams $B$ and $C$ for $(Y,S, y_0)$ and $(X,T, x_0)$ respectively, so that the unique maximal path of $B$ is $y_0$ and the unique maximal path of $C$ is $x_0$. The two points are in fact, the intersections of the tops of the K-R systems (see Definition \ref{towers}) associated to the Bratteli diagrams respectively. Then one can  construct an ordered premorphism matching  the map $\alpha$.

When the two systems are not necessarily minimal,  by the nice results of \cite{Tak2},  it is still possible to realize the two systems by Vershik maps on ordered Bratteli diagrams $B$ and $C$, respectively. However, we need to have some specific realizations $B$ and $C$ so that the modelling of $\alpha$ by $\mathcal V(f)$ for some ordered premorphism $f:B\rightarrow C$ will guarantee that 
$$\mathcal V(f)(X_C^{\rm max})\subseteq X_B^{\rm max}.$$
 Consequently, in terms of K-R systems, we need to have the intersections of the top of the K-R system associated to $C$ to be mapped by $\alpha$ into the intersection of the top of the K-R systems associated to $B$.  In other words, to model $\alpha$ by an ordered premorphism, we need to consider  the two systems as triples:
 $ (X,T, X_0)$ and $(Y,S,Y_0)$ that  
  $X_0$ and $Y_0$ are closed sets associated to $X_C^{\rm max}$ and $X_B^{\rm max}$ respectively with $ \alpha(X_0)\subseteq Y_0$. In this regard, the following arguments are   needed. 

\begin{definition}\label{towers}
Let $(X,T)$ be a zero dimensional dynamical system and $W\subseteq X$ be a closed set. A {\it Kakutani-Rokhlin} (K-R) partition for $(X,T,W)$ is a partition $$\{Z(k,j):\ 1\leq k\leq K, \ 1\leq j\leq J(k)\}$$ of clopen sets for $X$ such that 
\begin{enumerate}
\item $T(Z(k,j))=Z(k,j+1)$ for all $1\leq k\leq K$ and $1\leq j< J(k)$,
\item $T(\cup_{k=1}^K Z(k, J(k)))=\cup_{k=1}^K Z(k,1)$,
\item $W\subseteq \cup_{k=1}^K Z(k, J(k))$.
\end{enumerate}
The set $\cup_{k=1}^K Z(k,J(k))$ is called the {\it top} of the partition and $\cup_{k=1}^K Z(k,1)$ is its {\it base}. A {\it system } of K-R partitions $(\mathcal P_n)_{n=0}^\infty$ for $(X,T,W)$ is a sequence of K-R partitions in which for every $n\geq 0$, $\mathcal P_{n+1}$ is a refinement of $\mathcal P_n$, the top of $\mathcal P_{n+1}$ is contained in the top of $\mathcal P_n$, and $\cup_{n=0}^\infty\mathcal P_n$ is a base for the topology of $X$.
\end{definition}

\begin{definition}[\cite{Tak3}]\label{basic}
Let $(X,T)$ be a zero dimensional dynamical system. We say that a closed subset $W\subseteq X$ is a {\it quasi-section} set if every clopen neighberhood $U$ of $W$ is a complete $T$-section in the sense of \cite{med}, i.e.,  $U$ meets every $T$-orbit of $X$ at least once, equivalently, $\cup_{n\in\mathbb Z} T^n(U)=X$. 
\end{definition}

\begin{remark}
If $W$ is a quasi-section for $(X,T)$ and $U$ a clopen neighbourhood of $W$, then every point $x\in U$ is recurrent to $U$, i.e., there is $n\in\mathbb N$ such that $T^n(x)\in U$. In fact, since $U$ is a $T$-section we have 
$X=\cup_{n\in\mathbb Z} T^n(U)$ and so $X=\cup_{-N}^N T^n(U)$ for some $N\geq 1$. Then $$X=T^{-(N+1)}(X)=\bigcup_{n=-2N-1}^{-1} T^n(U).$$
In \cite{med} the notion of {\it basic set} is defined which is a quasi-section $W$ with the extra property that $W$ meets every $T$-orbit of $X$ at most once.  For more properties of basic sets, see \cite{Tak4}. 
\end{remark}

The following lemma together with  Propositions~\ref{perfect} and Corollary ~\ref{equiv2} may be considered as  an alternative proof for \cite[Theorem 1.1]{Tak3}. The proofs of Lemma ~\ref{finer} and Proposition ~\ref{perfect} have similar arguments as in the proofs of  \cite[Lemma 2.2]{Poon} and \cite[Theorem 1.1]{Tak3} but the proof of Corollary ~\ref{equiv2} uses  ordered premorphism arguments.

\begin{lemma}\label{finer}
Let $(X,T)$ be a zero dimensional dynamical system and $A\subseteq X$ be a non-empty clopen set which is a complete $T$-section. Let $\mathcal P$ be an arbitarry partition of $X$ into clopen sets. Then there is a K-R partition $\mathcal Q$ for $(X,T)$ such that the top of $\mathcal Q$ is $A$ and $\mathcal Q$ refines $\mathcal P$.
\end{lemma}
\begin{proof}
The proof for existence of a K-R partition $\mathcal Q$  that its top is $A$ is as \cite[Lemma 2.2]{Poon}. Then one can apply the method described in the proof of \cite[Lemma 3.1]{putnam89} to make   $\mathcal Q$  finer than the given $\mathcal P$.
\end{proof}

\begin{proposition}\label{perfect}
Let $(X,T)$ be a zero dimensional dynamical system and $W\subseteq X$ be a closed non-empty set. The following statements are equivalent.
\begin{enumerate}
\item \label{per} There is a perfect ordered Bratteli diagram $B=(V,E,\leq)$ and a conjugacy $\gamma: (X,T)\rightarrow (X_B, T_B)$ such that $\gamma(W)=X_B^{\rm max}$. 
\item \label{seq} There is a system of K-R partitions $\{\mathcal P_n\}_{n=0}^\infty$ for $(X, T,W)$ such that 
$$\bigcap_{n=0}^\infty Z_n=W$$ where $Z_n$ is the top of $\mathcal B_n$, $n\geq 0$.
\item \label{wbase} $W$ is a quasi-section for $(X,T)$.
\end{enumerate}
\end{proposition}
\begin{proof}

\begin{enumerate}

\item [(1)] $\Rightarrow$ (2)
 Let $B$ and $\gamma:X\rightarrow X_B$ be as in $(1)$. Let $(\mathcal Q_n)_{n=0}^\infty$ be the standard sequence of K-R partitions obtained from $B$ (see, e.g., the description preceding \cite[Proposition 3.11]{aeg21}), that is, 
$$\mathcal Q_0=\{X_B\},\ \ \mathcal Q_n=\{U(e_1, e_2, \ldots, e_n):\ (e_1, e_2, \ldots, e_n)\in E_{1,n}\}.$$ Let $W_n$ denote the top of $\mathcal Q_n$. Then $\bigcap_{n=0}^\infty W_n=X_B^{\rm max}$. Let 
$$\mathcal P_n=\{\gamma^{-1}(L):\ L\in\mathcal Q_n\},\ \ Z_n=\gamma^{-1}(W_n),\ n\geq 0.$$
Then $\{\mathcal P_n\}_{n=0}^\infty$ is a system of K-R partition for $(X,T)$ such that $Z_n$ is the top of $\mathcal P_n$ and
$$\bigcap_{n=0}^\infty Z_n=\bigcap_{n=0}^\infty\gamma^{-1}(W_n)=\gamma^{-1}(X_B^{\rm max})=W.$$
\item [(2)]$\Rightarrow$(1) is very similar to  the case of properly ordered Bratteli diagrams where for a given system of K-R partitions, an ordered Bratteli diagram $B$ is constructed and a natural homeomorphism $\gamma: X\rightarrow X_B$ is defined (see \cite[Section 4]{hps92} and the paragraph following \cite[Lemma 3.4]{aeg21}).
Observe that $B$ is perfect since $\gamma\circ T\circ \gamma^{-1}:X_B\rightarrow X_B$ is a homeomorphism extension of the Vershik map $T_B: X_B\setminus X_B^{\rm max}\rightarrow X_B\setminus X_B^{\rm min}$. 
\item [(2)]$\Rightarrow$(3) First note that the top (and the base) of every K-R partition of $(X,T)$ is a complete $T$-section. Thus every $Z_n$ is a complete $T$-section. Let $A$ be a clopen subset of $X$ with $W\subseteq A$. Since $\bigcap_{n=0}^\infty Z_n=W$, it follows that there is some $n\in\mathbb N$ with $Z_n\subseteq A$, since otherwise one can choose $x_n\in Z_n\setminus A$ for all $n\in\mathbb N$.   Passing to a subsequence, it can be assumed that $x_n\rightarrow x$ for some $x\in X$. Then  $Z_1\supseteq Z_2\supset\cdots$ implies that $x\in\bigcap_{n=1}^\infty Z_n=W$. But $\{x_n\}_{n=1}^\infty\subseteq X\setminus A$ and $A$ is clopen, so we get that $x\in X\setminus A$ contradicting $W\subseteq A$. Hence $A$ contains some complete $T$-section $Z_n$ and therefore $A$ is a complete $T$-section. 
\item [(3)]$\Rightarrow$ (2) If $W$ is a quasi-section and $\{Z_n\}_{n=0}^\infty$ is any decreasing sequence of clopen sets with $\bigcap_{n=0}^\infty Z_n=W$ (which exist as $X$ is zero dimensional) then using Lemma \ref{finer} repeatedly, we can construct  the desired sequence of K-R partitions. \qedhere
\end{enumerate}
\end{proof}

Now we have the tools for modelling a factoring map $\alpha: (X,T)\rightarrow (Y,S)$ in terms of ordered premorphisms.

Let $W$ be a quasi-section for $(X,T)$. A {\it Bratteli-Vershik realization} (B-V) of $(X,T,W)$  is a perfect ordered Bratteli diagram $B$ satisfying conditions (1)-(3) of Proposition \ref{perfect}.

\medskip

Now we present the proof of Theorem \ref{fact2}.
\begin{proof}[\bf Proof of Theorem \ref{fact2}.]
Let $(\mathcal P_n)_{n=0}^\infty$ and $(\mathcal Q_n)_{n=0}^\infty$ be the systems of K-R partitions for $(X,T,X_0)$ and $(Y,S,Y_0)$ supporting  the Bratteli-Vershik realizations $C=(V,E,\leq)$ and $B=(W,S,\leq)$, respectively. We proceed by the method described in the third paragraph after  \cite[Lemma 3.6]{aeg21} to obtain a cofinal increasing sequence $(f_n)_{n=0}^\infty$ in $\mathbb N\cup\{0\}$ and a sequence of edges $(F_n)_{n=0}^\infty$ leading to an ordered  premorphism $f=(F, (f_n)_{n=0}^\infty, \leq):B\rightarrow C$ such that $\mathcal V(f)=\gamma_2\circ\alpha\circ\gamma_1^{-1}$ where $\gamma_1:X\rightarrow X_B$ and $\gamma_2:Y\rightarrow Y_C$ are as in Proposition \ref{perfect}(1). The main point that the same method works here is that for each $\mathcal Q_n$, the induced partition $$\alpha^{-1}(\mathcal Q_n)=\{\alpha^{-1}(L):\ L\in\mathcal Q_n\}$$ of $X$ is a K-R partition such that its top contains $X_0$ as $\alpha(X_0)\subseteq Y_0$ and the top of $\mathcal Q_n$ contains $Y_0$. Since the intersection of the top of $\mathcal P_n$'s equals $X_0$, we may find a large enough $f_n\in\mathbb N$ such that $\mathcal P_{f_n}$ refines $\alpha^{-1}(\mathcal Q_n)$ and the top  of $\mathcal P_{f_n}$ is contained in the top of $\alpha^{-1}(\mathcal Q_n)$ (see the proof of (2)$\Rightarrow$(3) of Proposition \ref{perfect}). 
The uniqueness (up to equivalence) of $f$ follows from Proposition \ref{unique1}.
\end{proof}
\begin{corollary}
Let $(X,T)$ and $(Y,S)$ be two zero dimensional dynamical systems. If $\alpha: (X,T)\rightarrow (Y,S)$ is a topological factoring  then there are B-V realizations $C$ and $B$ for $(X,T)$ and $(Y,S)$ respectively such that $B$ and $C$ are perfect and there exists an ordered premorphism $f:B\rightarrow C$.
\end{corollary}
\begin{proof}
By \cite{Tak2}, $(X,T)$ has some B-V realization $C$.  Then by Proposition~\ref{perfect}  (or \cite[Theorem~1.1]{Tak3}), one can find  quasi-section $X_0$  associated to the set of infinite maximal paths of the Bratteli diagram.  By the topological factoring,  $Y_0=\alpha(X_0)$ is a quasi-section for $(Y,S)$. Now one can apply Theorem~\ref{fact2} to model $(Y, S, Y_0)$ by  an appropriate Bratteli diagram $B$ with an ordered premorphism $f:B\rightarrow C$.
\end{proof}

\begin{corollary}\label{equiv2}
Any two B-V realizations for a zero dimensional dynamical system $(X,T,W)$, where $W$ is a quasi-section, are equivalent. 
\end{corollary}
\begin{proof}
Let $B$ and $C$ be two B-V representations for $(X,T,W)$. Consider $$\alpha={\rm id}:(X,T)\rightarrow (X,T).$$
By Theorem~\ref{fact2}, there are ordered premorphisms $$f:B\rightarrow C,\ \ \ g:C\rightarrow B$$
such that $$\mathcal V(f)=\alpha={\rm id},\ \ \mathcal V(g)=\alpha^{-1}={\rm id}.$$
Then $gf:B\rightarrow B$ and $fg:C\rightarrow C$ are ordered premorphisms (see \cite[Definition 2.7]{aeg21} for composition of two ordered premorphisms) and $$\mathcal V(gf)=\mathcal V(f)\mathcal V(g)={\rm id}=\mathcal V({\rm id}_B),  \ \mathcal V(fg)=\mathcal V({\rm id}_C)$$
where ${\rm id}_B:B\rightarrow B$ and ${\rm id}_C:C\rightarrow C$ are the identity premorphisms. By Proposition~\ref{unique1}, 
$$gf\sim {\rm id}_B,\ \ \ fg\sim{\rm id}_C.$$
Therefore, $[g][f]=[{\rm id}_B]$ and $[f][g]=[{\rm id}_C]$. Thus $[f]:B\rightarrow C$ is an isomorphism of ordered Bratteli diagrams. It turns out that $B$ and $C$ are isomorphisc in the category of ordered Bratteli diagrams and hence they are equivalent by \cite[Proposition 2.9]{aeg21}. The proof is finished here. 
\medskip

There is an alternative proof for this corollary using a K-R partition argument.  Let $B=(V,E,\leq)$ and $C=(W, S, \leq)$ be two B-V realizations of $(X,T, W)$ obtained from K-R systems $\{\mathcal P_n\}_{n\geq 0}$ and $\{\mathcal Q_n\}_{n\geq 0}$, respectively. Let $Z_n$ and $W_n$ be the top levels of $\mathcal P_n$ and $\mathcal Q_n$, respectively for every $n\geq 0$. Set $n_0=0$ and $n_1=1$. Since 
$$\bigcap_{n=0}^\infty Z_n=W=\bigcap _{n=0}^\infty W_n$$
and $\cup_{n=1}^\infty\mathcal Q_n$ is a basis for the topology of $X$, there exists $n_2>n_1$ such that $\mathcal Q_{n_2}$ refines $\mathcal P_{n_1}$ and $W_{n_2}\subseteq Z_{n_1}$ (the latter follows from an argument similar to the one in the proof of (2)$\Rightarrow$ (3) of Proposition~\ref{perfect}). Similarly, there is $n_3>n_2$ such that $\mathcal P_{n_3}$ refines $\mathcal Q_{n_2}$ and $Z_{n_3}\subseteq W_{n_2}$. Continuing this procedure, we obtain a strictly increasing sequence $\{n_k\}_{k=0}^\infty$ such that
$$\mathcal P_{n_1}\geq \mathcal Q_{n_2}\geq \mathcal P_{n_3}\geq \mathcal Q_{n_4}\geq\cdots$$
and $$Z_{n_1}\supseteq W_{n_2}\supseteq Z_{n_3}\supseteq W_{n_4}\supseteq \cdots.$$
Put $R_0=\{X\}$, $R_k=\mathcal P_{n_k}$ for all odd $k$, and $R_k=\mathcal Q_{n_k}$ for every even $k$. Then $\{R_k\}_{k=0}^\infty$ is a K-R system for $(X,T,W)$ and gives an ordered Bratteli diagram $D$ which is a B-V realization for $(X,T,W)$ and telescoping it along odd  (resp., even) levels equals the telescoping of $B$ (resp., $C$) along $\{n_{2k+1}\}_{k=0}^\infty$ (resp., $\{n_{2k}\}_{k=0}^\infty$). Thus $B$ and $C$ are equivalent. 
\end{proof}
\begin{remark}\label{singleton1}
Let us recall that by the results of \cite{hps92, putnam89, Tak3}, for a zero dimensional dynamical system $(X,T)$, there exists a singleton quasi-section $\{x_0\}$ if and only if $(X,T)$ is essentially minimal. 
Indeed, when $\{x_0\}$ is a quasi-section for $(X,T)$ then there exists a perfect ordered Bratteli diagram $B$ with a conjugacy  $\gamma: (X,T)\rightarrow(X_B,T_B)$ that $\gamma(\{x_0\})=X_B^{\max}$ (see also Proposition \ref{perfect}). In particular, $X_B^{\max}$ is a singleton and therefore, $X_B^{\min}$ is a singleton. Thus $(X,T,x_0)$ is essentially minimal. Conversely, when $(X,T)$ is essentially minimal,  by \cite[Theorem 1.1]{hps92}, any point $x_0$ in the unique minimal subset, is a quasi-section.
\end{remark}
\section{Topological Factoring, Ordered Premopphisms and Inverse Limit Systems}
By the well-known theorem of Curtis-Hedlund-Lyndon, topological factoring between two subshift systems can be modelled by a local rule called {\it the sliding block code} between the two systems \cite{hedlund}. 
In this section, we go through the proof of  a generalization of this theorem for zero dimensional dynamical systems. We see in Theorem \ref{inversefactor} that in this general case, the factoring is defined by a sequence of sliding block codes.

Let $(X,T,X_0)$ be a zero dimensional dynamical system with a quasi-section $X_0$. Consider the sequence of K-R partitions $\{\mathcal Q_k\}_{k\geq 0}$ for $(X,T,X_0)$ as in Proposition \ref{perfect}(2), where $\mathcal Q_0=\{X\}$.  Then there exists a truncation map $\tau_k:X\rightarrow \mathcal Q_k$ defined by $\tau_k(x)=U$ where $U$ is the unique element in $\mathcal Q_k$ that $x\in U$. So the natural projections $\tilde{\tau_k}: X\rightarrow \mathcal Q_k^\Z$  are defined by 
 \begin{equation}\label{proj}
 \tilde{\tau}_k(x)=(\tau_k(T^nx))_{n\in\Z}.
 \end{equation}
  It turns out that at each level $k$, we have a subshift system $(\tilde{\mathcal Q}_k,\sigma)$, also known as a {\it symbolic factor} of $(X,T)$ with respect to the partition $\mathcal Q_k$:
$$(\tilde{\mathcal Q}_k,\sigma) \ \ {\rm where} \ \ \tilde{\mathcal Q}_k=\tilde{\tau}_k(X)\subseteq \mathcal Q_k^\mathbb Z, \ \ \sigma(\tilde{\tau}_k(x))=\tilde{\tau}_k(Tx).$$ 
As $\mathcal Q_k$ refines $\mathcal Q_{k-1}$, there is a natural map $\mathcal Q_k\rightarrow\mathcal Q_{k-1}$ sending $U\in\mathcal Q_k$ to $V\in\mathcal Q_{k-1}$ where $U\subseteq V$. This map can be considered as a $1$-block map inducing a sliding block code 
$\alpha_k: (\tilde{\mathcal Q}_k,\sigma)\rightarrow (\tilde{\mathcal Q}_{k-1},\sigma)$. 
Note that $\alpha_k\circ\tilde\tau_{k}=\tilde\tau_{k-1}$ for all $k\geq 1$, since
\begin{equation}\label{trunc1}
\alpha_k(\tilde\tau_k(x))=\alpha_k((\tau_k(T_B^n(x)))_{n\in\Z})=(\tau_{k-1}(T_B^n(x)))_{n\in\Z}=\tilde\tau_{k-1}(x).
\end{equation}
Consequently, we have the following inverse system whose inverse limit is conjugate to $(X,T)$:
\[
\xymatrix{(\tilde{\mathcal Q}_{0},\sigma)
&( \tilde{\mathcal Q}_{1},\sigma)\ar[l]_{\alpha_{1}} &(\tilde{\mathcal Q}_{2},\sigma)\ar[l]_{\alpha_{2}} &\cdots\ar[l]_{\alpha_3}&  (X,T,X_0)\ar[l]
}
\]
\begin{proof}[\bf Proof of Theorem \ref{inversefactor}.]  
First assume that $\pi:X\rightarrow Y$ is a topological factoring with $\pi(X_0)\subseteq Y_0$.  Consider the inverse limit systems associated to the two systems as described above. Suppose that  $d_1$ and $d_2$ are the metrics on $C$ and $B$, respectively that are compatible with the topologies on $X$ and $Y$. 
 By the proof of Theorem \ref{fact2},  there exists a  strictly increasing sequence of non-negative integers  $(f_n)_{n=0}^\infty$ that leads to the existence of  an ordered premorphism $f$ between the two Bratteli diagrams. 
So for every $i$, let $n_i:=f_i$ and  consider the sequence of K-R partitions $\{\mathcal Q_{n_i}\}_{i\geq 0}$.
Then the maps $\pi_k$'s can be well-defined  by using  the natural projections  $\tilde{\tau}_{n_k}:X\rightarrow \tilde{Q}_{n_k}$ and $\tilde\tau'_k:Y\rightarrow \tilde{P}_k$. Indeed, for every $k\geq 0$ we have
 $$\pi_k:\tilde{Q}_{n_k}\rightarrow \tilde{P}_k,\ \ \ \ \pi_k(\tilde{\tau}_{n_k}(x)):=\tilde\tau'_k(\pi (x))$$ which make topological factorings between the associated local subshifts. In other words, by (\ref{trunc1}),
 \begin{eqnarray*}
 \pi_k\circ\sigma(\tilde\tau_{n_k}(x))&=&\pi_k\tilde\tau_{n_k}\circ T(x)=\tilde\tau'_k(\pi(Tx))=\tilde\tau'_k(S\pi(x))\\
 &=&\sigma\circ\tilde\tau'_k(\pi(x))\\
 &=&\sigma\circ\pi_k(\tilde\tau_{n_k}(x)).
 \end{eqnarray*}
 Moreover, by (\ref{proj}) and (\ref{trunc1}), for every $k\geq 1$, 
 \begin{eqnarray*}
 \beta_k\circ\pi_k(\tilde\tau_{n_k}(x))&=&\beta_k(\tilde\tau'_k(\pi(x)))\\
 &=&\beta_k((\tau'_k(S^n(\pi(x))))_{n\in\Z})\\
&=&(\tau'_{k-1}(S^n(\pi(x))))_{n\in\Z}\\
&=&\tilde\tau'_{k-1}(\pi(x))\\
&=&\pi_{k-1}(\tilde{\tau}_{n_{k-1}}(x))\\
&=&\pi_{k-1}\circ\gamma_k(\tilde\tau_{n_k}(x)).
 \end{eqnarray*}
 For the other direction, assume that for  $(X,T,X_0)$, $(Y,S,Y_0)$ and their associated K-R partitions $\{\mathcal Q_n\}_{n\geq 0}$ and $\{\mathcal P_n\}_{n\geq 0}$ respectively,  there exists a sequence $\{n_i\}_{i\geq 0}$ such that  the Diagram~\ref{diag} exists and all the rectangles in that commute. Then it is straightforward that the map $$\pi:(X,T,X_0)\rightarrow (Y,S,Y_0)$$ defined by 
 $$ \pi(x):=\varprojlim_i \pi_i(\tilde\tau_{n_i}(x)),\ x\in X$$
  is a topological factoring and $\pi(X_0)\subseteq Y_0$. 
\end{proof}

We recall the S-adic representation of an ordered Bratteli diagram form  \cite{don} and \cite[Subsection 2.4]{gh18}. Let  $B=(V,E,\leq)$ be an ordered Bratteli diagram. Then $\sigma^B=(\sigma_i^B:V_i\rightarrow V_{i-1}^*)_{i\geq 1}$ is defined for $i\geq 2$ by 
$$ \sigma_i^B(v)=s(e_1(v))s(e_2(v))\cdots s(e_k(v))$$
where $\{e_j(v):\ j=1,\ldots, k(v)\}$ is the ordered set of the edges in $E_i$ with range $v$, and for $i=1$, $ \sigma_1^B:V_1^*\rightarrow E_1^*$, is defined by $\sigma_1^B(v)=e_1(v)\cdots e_\ell(v)$ where $e_1(v),\ldots, e_\ell(v)$ are all the edges in $E_1$ with range $v\in V_1$ and $e_1(v)<\cdots<e_\ell(v)$. Note that by concatenation, one can extend $\sigma_i^B$ as 
$\sigma_i^B:V_i^*\rightarrow V_{i-1}^*$. Also, recall that $\sigma_{(i,j]}^B=\sigma_{i+1}^B\circ\sigma_{i+2}^B\circ\cdots\circ\sigma_j^B$ is a morphism from $V_j^*$ to $V_i^*$ for $0\leq i\leq j$. We say that a morphism $\sigma:A^*\rightarrow B^*$ is letter-surjective if for any $b\in B$ there is $a\in A$ such that $b$ appears in $\sigma(a)$.

Now consider $(Y,S,Y_0)$ and $(X,T,X_0)$ with their associated  B-V models $B=(V,E,\leq)$ and $C=(W,E',\leq)$, respectively. Having the ordered premorphism $f=(F,(f_k)_{k=0}^\infty,\geq)$ (see Definition \ref{def61} for notations), for each $k\geq 1$ the set of edges $F_k$ induces a morphism  $\eta_k:W_{n_k}\rightarrow V_k^*$. To see this, suppose that $w\in W_{n_k}$. By the definition of $F_k$, there exists an ordered set of edges in  $F_k$, say $\{g_1, g_2, \cdots, g_m\}$ 
 such that for every $1\leq i\leq m$, $s(g_i)\in V_k$, i.e. the source of the edge $g_i$ is a vertex in $V_k$. Then
 \begin{equation}\label{commute}
 \eta_k(w)=s(g_1)s(g_2)\cdots s(g_m).
 \end{equation}
This can naturally be extended to $W_{n_k}^*$ by concatenation. Then ordered commutativity of the premorphism $f$ implies that
$$\forall k\geq 1,\ \ \eta_{k-1}\circ \sigma^C_{(n_{k-1},n_k]}=\sigma^B_k\circ \eta_k.$$
where $\sigma^B_i:V_{i+1}\rightarrow V_i^*$ and $\sigma^C_i:W_{i+1}\rightarrow W_i^*$, $i\geq 1$ are the morphisms between consecutive levels of the Bratteli diagrams $B$ and $C$ respectively and $\sigma^C_{(n_i,n_{i+1}]}=\sigma^C_{n_i+1}\circ\cdots\circ\sigma^C_{n_{i+1}}$. 
 In other words, we have the following sequence of commutative (rectangular) diagrams:
\begin{equation}\label{diagram}
\xymatrix{W_{0}\ar[d]_{\eta_{0}}
& W_{n_1}^*\ar[l]_{_{\sigma^C_{(0,n_1]}}}\ar[d]_{\eta_{1}} &W_{n_2}^*\ar[l]_{\sigma^C_{(n_1,n_2]}}\ar[d]_{\eta_{2}} &\cdots\ar[l]_{\xi_{(n_2,n_3]}}&  \ \  \\
\ V_{0}
& V_{1}^*\ar[l]^{\sigma^B_1} &V_{2}^*\ar[l]^{\sigma^B_2}&\cdots\ar[l]^{\sigma^B_3} \   
 }
\end{equation}

Note that  for every $i\geq 0$, the morphism $\eta_i:W_{n_i}\rightarrow V_i^*$ is in the opposite direction of $F_i:V_i\rightarrow W_{n_i}$ (used in the previous sections). In fact, they  essentially coincide, meaning that $\eta_i$ is the S-adic interpretation of $F_i$.

We know that each tower $\mathcal T$ in $\mathcal P_i$ (resp. $\mathcal Q_i$), $i\geq 1$ is associated with  a vertex $v\in V_i$ (resp. $w\in W_i$) and all the edges terminating at it from $V_{i-1}$ (resp. $W_{i-1}$). Therefore, for  each tower $\mathcal T\in\mathcal Q_{n_i}$, the morphism $\eta_k$  specifies a stacking of  $m$  towers of $\mathcal P_i$ as in equation (\ref{commute}).
Then we have the following proposition as a corollary of  Theorem 1.1  and Theorem 1.2.
\begin{proposition}\label{localmor}
Let $(X,T)$ and $(Y,S)$ be zero dimensional dynamical systems with quasi-sections $X_0$ and $\{y_0\}$ respectively. Then there exists a toplogical factoring $\pi:X\rightarrow Y$ with $\pi(X_0)=\{y_0\}$ if and only if  
for every B-V models $C=(W,E',\leq)$ and $B=(V,E,\leq)$ for $(X,T,X_0)$ and $(Y,S,\{y_0\})$ respectively,   there exists  an increasing sequence $\{n_i\}_{i\geq 0}$ of non-negative integers with $n_0=0$, and non-erasing letter-surjective morphisms $\eta_i:V_{n_i}^*\rightarrow W_i^*$  for every $i\geq 0$,  the following diagram commutes:
\begin{equation}\label{diagram}
\xymatrix{W_{n_i}^*\ar[d]_{\eta_{i}}
& \ \ W_{n_{i+1}}^*\ar[l]_{_{\sigma^C_{(n_i,n_{i+1}]}}}\ar[d]_{\eta_{i+1}} &\ \  \\
\ V_{i}^*
& \ \ V_{i+1}^*.\ar[l]^{\sigma^B_{i+1}} &\   
 }
\end{equation}
\end{proposition}


{\bf Acknowledgements.} We are grateful to Reem Yassawi for helpful discussions and her  comments on the final version of the introduction. We would also like to thank Bastian Espinoza for introducing \cite{Tak1} to us which led us to go through more recent publications of T. Shimomura, specifically \cite{Tak2, Tak4, Tak3}. 

The research of the first author was in part supported by a grant from IPM (No.1402460117) and a grant from INSF (4029595). The research of the second author  was fully supported by the EPSRC grant number EP/V007459/2. The project initiated  during the second author's appointment at The Open University.


\end{document}